\documentclass[11pt]{amsart}

\usepackage{amsmath, amsthm}
\usepackage{amssymb}
\usepackage{amsfonts}
\usepackage{latexsym}
\usepackage{mathrsfs}
\usepackage{bm}
\usepackage{graphicx}

\newcommand{\M}{\mathbb{M}}

\DeclareMathOperator{\im}{Im}
\DeclareMathOperator{\re}{Re}
\renewcommand{\Re}{\re}
\renewcommand{\Im}{\im}

\newcommand{\Hy}{\mathbb{H}}

\renewcommand{\tilde}[1]{\widetilde{#1}}

\newcommand{\D}{\mathcal{D}}

\newcommand{\tec}{Teichm\"uller }
\newcommand{\wep}{Weil-Petersson }

\renewcommand{\leq}{\leqslant}
\renewcommand{\geq}{\geqslant}

\DeclareMathOperator{\Prob}{Prob}
\DeclareMathOperator{\Area}{Area}

\DeclareMathOperator{\Dim}{dim}

\DeclareMathOperator{\dist}{dist}
\DeclareMathOperator{\inj}{inj}
\DeclareMathOperator{\systole}{systole}
\DeclareMathOperator{\Aut}{Aut}
\DeclareMathOperator{\Teich}{Teich}
\DeclareMathOperator{\Te}{Teich}
\DeclareMathOperator{\Diff}{Diff}
\DeclareMathOperator{\Span}{Span}
\DeclareMathOperator{\Sca}{Sca^{WP}}

\newcommand{\Mod}{\mbox{\rm Mod}}

\newcommand{\param}{{\mathchoice{\mkern1mu\mbox{\raise2.2pt\hbox{$\centerdot$}}\mkern1mu}{\mkern1mu\mbox{\raise2.2pt\hbox{$\centerdot$}}\mkern1mu}{\mkern1.5mu\centerdot\mkern1.5mu}{\mkern1.5mu\centerdot\mkern1.5mu}}}

\numberwithin{equation}{section}

\theoremstyle{plain}
\newtheorem{theorem}{Theorem}[section]

\newtheorem{corollary}[theorem]{Corollary}
\newtheorem{lemma}[theorem]{Lemma}
\newtheorem{proposition}[theorem]{Proposition}

\newtheorem*{acknowledgement}{Acknowledgement}

\theoremstyle{definition}
\newtheorem{definition}[theorem]{Definition}
\newtheorem{question}{Question}

\newtheorem{remark}[theorem]{Remark}
\theoremstyle{definition}
\newtheorem*{remarksenv}{Remarks}
\newtheorem*{rem}{Remark}

\begin{document}



\title[Uniform Bounds]
{Uniform Bounds for Weil-Petersson Curvatures}

\author{Michael Wolf \& Yunhui Wu}

\address{Department of Mathematics\\
       Rice University\\
         Houston, Texas, 77005-1892\\}
\email{mwolf@rice.edu}

\address{Yau Mathematical Sciences Center, Tsinghua University, Haidian District, Beijing 100084, China
}
\email{yunhui\_wu@mail.tsinghua.edu.cn  \& yhwu@math.tsinghua.edu.cn}

\begin{abstract}
We find bounds for Weil-Petersson holomorphic sectional
curvature, and the Weil-Petersson curvature operator in several regimes, that do not depend on the topology of the underlying surface. Among other results, we show that the minimal Weil-Petersson holomorphic sectional curvature of a sufficiently thick hyperbolic surface is comparable to $-1$, independently of the genus. This provides a counterexample to some suggestions that the Weil-Petersson metric becomes asymptotically flat, as the genus $g$ goes to infinity, in the thick loci in \tec space. Adopting a different perspective on curvature, we also show that the minimal (most negative) eigenvalue of the curvature operator at any point in the \tec space $\Teich(S_g)$ of a closed surface $S_g$ of genus $g$ is uniformly bounded away from zero. Restricting to a thick part of $\Teich(S_g)$, we show that the minimal eigenvalue is uniformly bounded below by an explicit constant which does not depend on the topology of the surface but only on the given bound on injectivity radius. 
\end{abstract}


\maketitle

\section{Introduction}
Let $S_g$ be a closed surface of genus $g$ where $g>1$, and $\Teich(S_g)$ be the Teichm\"uller space of $S_g$. Endowed with the \wep metric, the \tec space $\Teich(S_g)$ is K\"ahler (\cite{Ahlfors}), incomplete (\cite{Chu, Wol75}), geodesically complete (\cite{Wol87}) and negatively curved (\cite{Tromba86, Wol86}). Tromba \cite{Tromba86} and Wolpert \cite{Wol86} found a formula for the \wep curvature tensor, which has been applied to study a variety of curvature properties of $\Teich(S_g)$ over the past several decades. (See also  \cite{Jost91, LSY04, Siu86} for alternative proofs of the curvature formula.)


In their papers, they deduced from their formula that the holomorphic sectional curvatures are bounded above by a negative number which only depends on the genus of the surface, confirming a conjecture of Royden.

We focus in this paper on bounds for \wep curvatures that are uniform across \tec spaces, in the sense that the bounds do not depend on the genus $g$ of the surface $S_g$.

For example, in the second part of the paper -- described in more detail in the fifth subsection of this introduction -- we consider holomorphic sectional curvatures.  We find, for hyperbolic surfaces $X_g$ of genus $g$ of sufficiently large injectivity radius, that we may exhibit holomorphic sections whose curvatures are uniformly bounded away from zero; here the bound does not depend on the genus of the surface. 

In particular, we prove

\begin{theorem}\label{ubfhsc-0}
There is a constant $\iota_0$ so that if $X \in \Teich(S_g)$ is sufficiently thick, i.e. $\inj(X) > \iota_0$, then there exists a holomorphic section $\mu_X $ of the tangent space $T_X\Teich(S_g)$ so that the \wep holomorphic sectional curvature $K(\mu_X)$ along $\mu_X$ satisfies 

\[ K(\mu_X) \leq C<0\]
where $C$ is an explicit constant. 
\end{theorem}

Our interest in this bound came from suggestions that, as the genus $g$ increased without bound, there should be some sense in which the \wep metric became flatter at some loci in the moduli space, perhaps where the injectivity radii were becoming large.  Our example provides a counterexample to some conjectures in this direction; we provide more context and details for these two results later in this introduction.

Naturally, there are also a number of other appealing notions of curvature of a complex Riemannian manifold beyond holomorphic sectional curvatures. In particular, while we focus on the holomorphic sectional curvatures in the second part of the paper, in the first part of this paper we study the \wep curvature {\it operator} 
$$\tilde{Q}: \wedge^{2}T_X\Teich(S_g) \to \wedge^{2}T_X\Teich(S_g),$$ an endomorphism of the $(3g-3)(6g-7)$-dimensional exterior wedge product 
space $\wedge^{2}T_X\Teich(S_g)$ of the tangent space $T_X\Teich(S_g)$ at $X \in \Teich(S_g)$. Of course, one popular (and important) way to interpret the ``curvature" of a manifold is through its sectional curvatures: while (if we pick a basis of $\wedge^{2}T_X\Teich(S_g)$ consisting of simple vectors)  these sectional curvatures arise as diagonal elements of the curvature operator $\tilde{Q}$, our interest in the first part of the paper is on the full operator $\tilde{Q}$, at least in terms of estimating its spectrum. Our perspective is that a focus on the curvature operator allows for an appreciation of aspects of the \wep curvature that could be missing from discussions restricted to specializations of it, for example by just a focus on \wep sectional curvatures.

Now, in \cite{Wu14-co}
one of us 
showed that the Riemannian curvature operator $\tilde{Q}$ of the Weil-Petersson metric is non-positive definite.
In the first part of this paper, we develop some results for the negative eigenvalues of $\tilde{Q}$, especially as the genus of the surfaces represented in the moduli space becomes arbitrarily large. 

As an example of the sort of estimates we obtain, in focusing on regions in \tec space where the injectivity radius and the genus of the surface grows increasingly large, two questions collide.  The first asks, as above, whether the \wep metric grows flatter somewhere in some sense.  The other, noting that the scalar curvature of the \wep metric is growing comparably to $-g$, asks whether the most negative eigenvalue of the curvature operator $\tilde{Q}$ grows without bound.  In {\bf Theorem~\ref{largei-0}} below, we combine and refine Theorems~\ref{uubf-0} and \ref{lbfme-0} to simultaneously give a strong negative answer to both questions. 

We survey and describe the context for these and other results on the asymptotics of the curvature operator $\tilde{Q}$ in the next four subsections.
After those subsections, we will turn our attention to the discussion surrounding Theorem~\ref{ubfhsc-0}, beginning in subsection~\ref{subsec:bddHolCurvs-Intro}.

\subsection{Uniform lower bounds on the norm of the curvature operator.} 
We begin with the question, commented on above, whether as the genus $g$ of $S_g$ grows large, do some regions in the \tec space $\Teich(S_g)$
become increasingly flat?  
We show that, from the point of view of the full curvature operator, the answer is negative: on any sequence of surfaces with genus tending to infinity, at least the most negative eigenvalue of $\tilde{Q}$ does not tend to zero. Let $\lambda_{\min}(X)$ denote the minimal eigenvalue of $\tilde{Q}$ at $X$. We prove  

\begin{theorem}\label{uubf-0}
For any $X\in \Teich(S_g)$ and let $\Sca(X)$ be the Weil-Petersson scalar curvature at $X$, then 
\[\lambda_{\min}(X)\leq \frac{-1}{2\pi}<0.\]
More precisely,
\[\lambda_{{\min}}(X)\leq \frac{2 \Sca(X)}{9(g-1)}.\]
\end{theorem}


Because this eigenvalue is the minimum 
of values of $\tilde{Q}$ applied to (unit) elements of $\wedge^2T_{X}\Teich(S_g)$, it is enough to estimate $\tilde{Q}$ on a carefully chosen element of $\wedge^2T_{X}\Teich(S_g)$ which yields the bound.\newline

\subsection{Uniform upper bounds on the norm of the curvature operator on the thick locus.} 
It is often convenient to choose a basis for $\wedge^2T_{X}\Teich(S_g)$ comprised of simple two-vectors $e_i \wedge e_j$. Then since the minimal eignevalue of $\tilde{Q}$ is no more than the minimal value of a diagonal entry (note here that since $\tilde{Q}$ is non-positive definite, we are often comparing {\it negative} values), and with respect to this basis, the diagonal entries are the sectional curvatures, thus $\lambda_{\min}(X)$ is less than or equal to the sectional curvatures at $X$ along arbitrary planes. 

However, it is well-known that the sectional curvatures of $\Teich(S_g)$ along certain planes could be arbitrarily negative. One specific example is the direction along which a nontrivial simple closed curve of $S$ is pinching to zero. Then that sectional curvature, along the holomorphic plane defined by the pinching direction, tends to negative infinity as the curve goes to zero. More precisely, we let $\alpha$ be a nontrivial simple closed curve on $S$ and $\l_{\alpha}(X)$ be the length of the closed geodesic in $X$ representing $\alpha$. Consider the gradient $\lambda_\alpha:=grad (l_{\alpha}^{\frac{1}{2}})$. Wolpert in \cite{Wolpert5} proved that the magnitude of $\lambda_\alpha$ approaches $\frac{1}{\sqrt{4\pi}}$ as $l_\alpha$ goes to zero and the sectional curvatures along holomorphic planes spanned by $\lambda_\alpha$ behave as $\frac{-3}{\pi l_\alpha}+O(l_\alpha)$, which goes to negative infinity as $\alpha$ pinches to zero. One can also see a similar estimation by Huang in \cite{Huang1}. For more details, one can see Corollary 16 in \cite{Wolpert5}. Thus, in general we do not have any uniform lower bounds for $\lambda_{\min}(X)$ over $\Teich(S_g)$.

On the other hand, the Mumford compactness theorem \cite{Mumford71} implies that the thick part of the moduli space is compact. Since the mapping class group acts on $\Teich(S_g)$ by isometries, the Weil-Petersson curvature operator $\tilde{Q}$, restricted to the thick part of $\Teich(S_g)$, is bounded for each individual choice of $g$. In particular the minimal eigenvalue $\lambda_{\min}(X)$ has a lower bound when $X$ runs over the thick part of $\Teich(S_g)$. Our next result implies that the lower bound on the thick part may be taken to be uniform, independent of the topology of the surface.

\begin{theorem}\label{lbfme-0}
Given an $\epsilon>0$, let $\Teich(S_g)^{\geq \epsilon}$ be the $\epsilon$-thick part of $\Teich(S_g)$. Then, there exists a constant $B(\epsilon)>0$, depending only on $\epsilon$ (i.e. independent of the choice of genus $g$), such that
\[\lambda_{\min}(X) \geq -B(\epsilon), \ \  \forall X\in \Teich(S_g)^{\geq \epsilon}. \]
\end{theorem}

The argument will present an explicit constant $B(\epsilon)$ whose asymptotics as $\epsilon \to 0,\infty$ we will exploit in later sections of the paper. In addition, a direct corollary is that the sectional curvature of the Weil-Petersson metric, restricted to the thick part of the moduli space, is uniformly bounded from below, a result due originally to Huang \cite{Huang}.\newline

\subsection{Uniform pinched bounds on the norm of the curvature operator on sequences of increasingly thick surfaces.}\label{sec:intro-uniform bounds- thick}


The subject of the asymptotic geometry of $\mathbb{M}_g$, the moduli space of $S_g$, as $g$ tends to infinity, has recently become quite active: see for example Mirzakhani \cite{MM1, MM2, MM3, MM4} and Cavendish-Parlier \cite{CP12} (results obtained by refining Brock's \cite{Brock03} quasi-isometry of $\Teich(S_g)$ to the pants graph).  In terms of curvature bounds, by combining the results in Wolpert \cite{Wol86} and Teo \cite{Teo08}, we may easily see that, restricted on the thick part of the moduli space, the scalar curvature $\Sca$ is comparable to $-g$ as $g$ goes to infinity. 
Of course, the \wep scalar curvature $\Sca$ is a sum of $(3g-3)(6g-7)$ negative numbers, begging the question of whether the blowup $\Sca \asymp -g$ is attributable to the large and growing number of sectional curvatures in the scalar curvature sum, or whether the smallest sectional curvatures are themselves tending to $-\infty$ without bound.
In the sixth section of this paper, we study the asymptotic behaviors of certain eigenvalues of $\tilde{Q}$ as the genus goes to infinity.    

We recall for context that Buser and Sarnak \cite{BS94} proved  that there exists a family of closed hyperbolic surfaces $X_{k}$ of genus $g_k$ with $g_k \to \infty$ as $k \to \infty$ and $\inj(X_k)\geq \frac{2\ln{g_k}}{3}$. Now, we noted above that the constant $B(\epsilon)$ in Theorem~\ref{lbfme-0} may be explicitly stated. By examining the asymptotic property of that constant $B(\epsilon)$ as $\epsilon \to \infty$, we prove that the minimal eigenvalue of the Weil-Petersson curvature operator on a surface of large injectivity radius is pinched by two explicit negative numbers.

\begin{theorem}\label{largei-0}
Let $\{X_{g}\}$ be a sequence of hyperbolic surfaces whose injectivity radii satisfy $\lim_{g\to \infty}\inj(X_g)=\infty$. Then, for $g$ large enough, 
\[ \frac{-25}{\pi}\leq \lambda_{\min}(X_g) \leq \frac{-1}{2\pi}.\]
\end{theorem}
\

 
We refine this analysis of the family of eigenvalues further in the next passage.


\subsection{Bounds for eigenvalues in the thick part of $\M_g$ for large genus $g$.} Thus far, we have focused exclusively on the minimal eigenvalue $\lambda_{\min}(X_g)$.  In the latter portion of this part of the paper, we study some other eigenvalues of the Weil-Petersson curvature operator and their dependence on the genus $g$. In section~ \ref{s-coos}, we restate Theorem 1.1 in \cite{Wu14-co}  in terms of the eigenvalues of the Weil-Petersson curvature operator. More precisely,
\begin{theorem}\label{(3g-3)^2-0}
For any $X \in \Teich(S_g)$, the Weil-Petersson curvature operator $\tilde{Q}$ at $X$ has exactly $(3g-3)^2$ negative eigenvalues.
\end{theorem}

Thus, for any $X\in \Teich(S_g)$, one may list the set of all non-zero eigenvalues of the Weil-Petersson curvature operator at $X$ as follows
\[\lambda_{(3g-3)^2}(X)\leq \lambda_{(3g-3)^2-1}(X)\leq \cdots \lambda_{2}(X)\leq \lambda_{1}(X)<0.\]
The minimal eigenvalue $\lambda_{\min}(X)=\lambda_{(3g-3)^2}(X)$ basically measures the norm of the curvature operator $\tilde{Q}$ at $X$.

For our next result, we interpret the bulk asymptotics of these collection of eigenvalues in terms of their $\ell^p$ norms. This will allow us to provide a family of estimates that bridge what is already known with our results in Theorems~\ref{uubf-0} and \ref{lbfme-0}.

We define the \textsl{$\ell^p$-norm $||\tilde{Q}||_{\ell^p}(X)$} of the Weil-Petersson curvature operator at $X$ as
\[||\tilde{Q}||_{\ell^p}(X):=(\sum_{i=1}^{(3g-3)^2}{|\lambda_i(X)|^p})^{\frac{1}{p}} \ (1\leq p\leq \infty).\]

Assume that $\{X_g\}_{g\geq 1}$ is a sequence of hyperbolic surfaces with $\inj(X_g)\geq \epsilon>0$ for any given positive constant $\epsilon$. \\

(1). When $p=1$, by the results in Wolpert \cite{Wol86} and Teo \cite{Teo08} we know that $||\tilde{Q}||_{\ell^1}(X_g)$=$-\Sca(X_g)$ is comparable to $g$, as $g$ goes to infinity.\\

(2). When $p=\infty$, by Theorem \ref{uubf-0} and \ref{lbfme-0} we know that $||\tilde{Q}||_{\ell^\infty}(X_g)=-\lambda_{(3g-3)^2}(X_g)$ is comparable to 1.\\

For $1<p<\infty$, we have the following upper bound.

\begin{theorem}\label{ubfl^p}
Given an $\epsilon>0$, let $\Teich(S_g)^{\geq \epsilon}$ be the $\epsilon$-thick part of $\Teich(S_g)$. Then, for any $1<p< \infty$, there exists a constant $B(\epsilon)>0$ (the same constant in Theorem \ref{lbfme-0}), depending only on $\epsilon$, such that the $\ell^p$-norm  of the Weil-Petersson curvature operator satisfies
\[||\tilde{Q}||_{\ell^p}(X_g) \leq B(\epsilon) \cdot g^{\frac{1}{p}}, \ \  \forall X_g \in \Teich(S_g)^{\geq \epsilon}. \]
\end{theorem}

\begin{rem}
It is not known that \textsl{whether $||\tilde{Q}||_{\ell^p}(X_g)$ is comparable to $g^{\frac{1}{p}}$ as $g$ goes to infinity} except in the cases $p=1, \infty$. 
\end{rem}

Fixing $i \in \{1,2,\cdots,(3g-3)^2\}$, the $i$-th eigenvalue $\lambda_{i}(X)$ is a continuous function on $\Teich(S_g)$ since the Weil-Petersson metric is smooth. Thus, the Mumford compactness theorem implies that the function $\lambda_{i}(X)$ achieves its minima and maxima in the thick part $\Teich(S_g)^{\geq \epsilon}$ for fixed $\epsilon>0$;
we denote those minima and maxima by $\underline{\lambda}^{\epsilon}_{i}(g)$ and $\overline{\lambda}^{\epsilon}_{i}(g)$ respectively. In this notation, focusing our attention now on the index $i$, Theorem \ref{uubf-0} and Theorem \ref{lbfme-0} state that both $\underline{\lambda}^{\epsilon}_{(3g-3)^2}(g)$ and $\overline{\lambda}^{\epsilon}_{(3g-3)^2}(g)$ are pinched by two negative numbers which are independent of the genus of the surface. In section 6 we show that, if that index $i$ is not close to $(3g-3)^2$, then both $\underline{\lambda}^{\epsilon}_{i}(g)$ and $\overline{\lambda}^{\epsilon}_{i}(g)$ could be arbitrarily close to zero, once we take $g$ large enough. 

More precisely,  we consider the $i^{th}$ eigenvalue $\lambda_i$ of $\tilde{Q}$.  Suppose that this index $i \in \{1,2,3,..., (3g-3)^2\}$ for this eigenvalue is one of the earlier indices in the sense that $\frac{i(g)}{9g^2} = \alpha < 1$ (for example, $i \notin \{9g^2 - 6g - 1, 9g^2 - 6g, 9g^2 -6g+1, ...., (3g-3)^2\}$).

\begin{theorem}\label{fftz-0}
Let $\lambda_i$ be an eigenvalue of $\tilde{Q}$ with index $i = i(g)$ satisfying the condition $\limsup_{g\to \infty} \frac{i}{9g^2} = \alpha < 1$ just above.
Then, for $g$ sufficiently large, we have
\[ \frac{-B(\alpha,\epsilon)}{g}\leq \underline{\lambda}^{\epsilon}_{i(g)}(g)\leq  \overline{\lambda}^{\epsilon}_{i(g)}(g)<0\]
where $B(\alpha,\epsilon)$ is a constant only depending on $\alpha$ and $\epsilon$. In particular,
\[\lim_{g\to \infty} \underline{\lambda}^{\epsilon}_{i}(g)=\overline{\lambda}^{\epsilon}_{i}(g)=0, \ \  for \ all \ \ 1\leq i \leq 8g^2.\] 
\end{theorem}

A direct consequence is 
\begin{corollary}\label{poffm}
Fix $\epsilon_0>0$, then for any $\epsilon>0$, the probability
\[\Prob\{1\leq i \leq (3g-3)^2; \ \ |\underline{\lambda}^{\epsilon_0}_{i}(g)|<\epsilon\}\to 1 \ as \ g \to \infty.\]
\end{corollary}
\begin{proof}
Let $p(g):=\Prob\{1\leq i \leq (3g-3)^2; \ \ |\underline{\lambda}^{\epsilon_0}_{i}(g)|<\epsilon\}$. Theorem~\ref{fftz-0} then says that, for any $0\leq \alpha<1$,  
\[\alpha \leq \liminf_{g\to \infty}p(g)\leq \limsup_{g\to \infty}p(g)\leq 1.\]
Since $\alpha$ is arbitrary in $[0,1)$, $\lim_{g\to \infty}p(g)=1.$
\end{proof}

Mirzakhani (Theorem 4.2 in \cite{MM4}) proved that for a small enough number $\epsilon_0>0$, the volume of the $\epsilon_0$-thin part of the moduli space $\mathbb{M}_g$ is comparable to $\epsilon_0^2 Vol(\mathbb{M}_g)$ as $g \to \infty$. Corollary~\ref{poffm} suggests that the moduli space $\mathbb{M}_g$ tends to be flat as $g$ goes to infinity in a probabilistic sense. \newline

\subsection{Existence of holomorphic lines on sufficiently thick surfaces with uniformly pinched Weil-Petersson holomorphic sectional curvatures.}\label{subsec:bddHolCurvs-Intro}
We now provide a bit more discussion and a clarification of Theorem~\ref{ubfhsc-0} that we stated in the first overview subsection.

In \cite{Tromba86, Wol86} it was shown that the sectional curvature of the Weil-Petersson metric for $\Teich(S_g)$ is negative and the holomorphic sectional curvature is bounded above by a negative number comparable to $-\frac{1}{g}$. The dependence of this bound on $g$ begs the question as to whether there are bounds on the sectional curvature that are independent of the topology of the surface, even if one allows an additional restriction to a thick part of moduli space. However, Teo in \cite{Teo08} showed that, restricted to any thick part of the moduli space, the Ricci curvature is uniformly bounded from below. As the Ricci curvature is a trace over $6g-7$ curvatures, we then see that some sectional curvatures go to zero along any sequence of surfaces $X_g$ in the thick part of $\Teich(S_g)$ as $g \to \infty$. This also suggests the question\footnote{We are told by Zheng Huang that Question \ref{qom} is originally raised by Maryam Mirzakhani. We would like to thank both of them here.} whether,
\begin{question}\label{qom}
Restricted to any thick parts of the moduli spaces, do all of the sectional curvatures tend to zero as the genus goes to infinity?
\end{question}

The goal of the second part of this paper is to show, for the portion of the \tec space composed of surfaces which are ``sufficiently thick" -- a term we will define precisely in Definition~\ref{defn: suff thick} -- 
the existence of holomorphic lines whose 
holomorphic sectional curvatures are uniformly bounded away from zero. This result, partially stated as Theorem~\ref{ubfhsc-0} in the opening paragraphs of the introduction, gives a negative answer to the question above. 

In notational preparation for the more precise version of the result, we recall that a holomorphic section of the (complexified) tangent space $T_X\Teich(S_g)$
of $X \in \Teich(S_g)$ in the \tec space $\Teich(S_g)$ is defined by a harmonic Beltrami differential $\mu$. Then Theorem~\ref{ubfhsc-0} may be (slightly) refined to read as

\begin{theorem}\label{ubfhsc}
There is a constant $\iota_0$ so that if $X \in \Teich(S_g)$ is sufficiently thick, i.e. $\inj(X) > \iota_0$, then there exists a $\mu_X $ in the tangent space $T_X\Teich(S_g)$ so that the \wep holomorphic sectional curvature $K(\mu_X)$ along $\mu_X$ satisfies 

\[ K(\mu_X) \leq \frac{-81C_0}{6400\cdot \pi^2}<0\]
where $C_0=\frac{2}{3C(1)^2\cdot Vol_{\D}(B(0;1))}$. 
\end{theorem}
This is the form of the theorem we shall prove in the final Section~\ref{sec:curv of hol lines} of this paper.

The expressions $C(1)$ refers to an explicit function $C(\inj(X))$ of the injectivity radius $\inj(X)$ which we shall display in Definition~\ref{defn:C-inj}. In addition, there is also a uniform lower bound for $K(\mu_X)$ which  is due to Huang in \cite{Huang} (see also Theorem \ref{lbfme-0}).

An almost immediate corollary of our method of proof refers to a sequence of surfaces $X_g$ of growing genus whose injectivity radii $\inj(X_g)\to \infty$ as $g \to \infty$.

\begin{corollary} \label{cor:bdd curv on seq}
Let $\{X_{g}\}$ be a sequence of hyperbolic surfaces whose injectivity radii satisfy $\lim_{g\to \infty}\inj(X_g)=\infty$. Then, there exists a uniform constant $E>0$ such that for $g$ large enough, the \wep holomorphic sectional curvatures satisfy
\[-\frac{2}{\pi}\leq\min_{\sigma_g \subset T_{X_g}\Teich(S_g)}K(\sigma_g)\leq -E<0\]
where the minimum runs over all the holomorphic lines in $T_{X_g}\Teich(S_g)$.
\end{corollary}

[We note that the above Theorem~\ref{ubfhsc} also proves that the minimal (most negative) eigenvalue of the Weil-Petersson curvature operator, restricted on sufficiently thick hyperbolic surface, is uniformly bounded away from zero. In that sense, this result then makes contact with, and is certainly consistent with, Theorem \ref{uubf-0}, but the two results are distinct, since our example here is restricted to hold on this special `sufficiently thick' region (and, furthermore, the asserted constants are different).]

\subsection{A remark on choice of normalizations.}
Some of the uniformity of the bounds that are independent of genus of course results from the conventions of fixing the uniformized metrics on the Riemann surfaces to all have curvature identically $-1$ (instead of, say, curvature identically $-g$), forcing the areas of the surfaces to grow linearly with $g$ in area.  Other conventions on representatives of the moduli space would translate into other bounds, but the ones we choose are the ones that seem most prevalent in the present literature.

\subsection{Discussion of the methods.} The methods for the two parts of the paper are distinct. In the first part, we begin with the formula of the Weil-Petersson curvature operator $\tilde{Q}$ in \cite{Wu14-co}: see equation \eqref{eqn:Q(A,A)decomposition}.  That formula exhibits $<\tilde{Q}(A), A>$ as  a sum of terms involving the operator $-2(\Delta-2)^{-1}$ and algebraic expressions in a holomorphic orthonormal basis of $T_{X}\Teich(S_g)$. By applying some fundamental properties of the operator $-2(\Delta-2)^{-1}$, we give lower and upper bounds for $<\tilde{Q}(A), A>$ in terms of, for our applications, quantities that only involve pointwise values of the holomorphic basis. Here we briefly outline our proofs of Theorem \ref{uubf-0}, Theorem \ref{lbfme-0} and Theorem~\ref{ubfhsc}.

For the proof of \textbf{Theorem \ref{uubf-0}}: By applying our upper bound for the curvature operator formula to the most symmetric element $A \in \wedge^2T_X\Teich(S_g)$ we show that (up to a positive constant) $<\tilde{Q}(A), A>$ is bounded above by the scalar curvature of the \wep metric at $X$, scaled by a factor of $\frac{1}{g}$. Then we apply Wolpert's upper bound of the scalar curvature which involves a factor of $g$ to finish the proof. 

For the proof of \textbf{Theorem \ref{lbfme-0}}: Choose an arbitrary element $B$ of unit length in $\wedge^2T_X\Teich(S_g)$. Our lower bound of $<\tilde{Q}(B),B>$ involves only terms that are products of our holomorphic basis, weighted by the coefficients of the element $B$ in that basis. Well-known bounds on terms like these then yield a bound in terms of those coefficients and an explicit function of the injectivity radius. 
The condition that $B$ is of unit norm then yields an expression that is independent of genus.  

For the proof of \textbf{Theorem~\ref{ubfhsc-0}}/\textbf{Theorem~\ref{ubfhsc}}: We choose a holomorphic line that corresponds to the image of a constant function on the disk under the $\Theta$-operator to the hyperbolic surface $X_g$, here thought of as the quotient of $\Hy^2$ under the action of a Fuchsian group $\Gamma_g$. The resulting Poincar\'e series has a term corresponding to the identity element of $\Gamma_g$ and a series of other `error' terms.  We adapt a method of Ahlfors \cite{Ahlfors64} to show that outside a large ball, the contribution of the error terms is bounded by a subharmonic function.   Taking the ball large enough, we can show that the contribution of the error terms, now estimated by its value on the boundary of the large ball -- in its role as the boundary of the region of subharmonicity -- can be made arbitrarily small. Thus only the term corresponding to the identity element of $\Gamma_g$ cannot be made arbitrarily small, and estimates of that term yield the desired bound.

\subsection{Plan of the paper.} Section \ref{s-np} provides some necessary background and the basic properties of the Weil-Petersson geometry of \tec space that we will need. After that preparatory section, the paper splits into two parts, each of which may be read independently of the other.  The first part treats our results on the \wep curvature operator, while the second part discusses the curvature of some holomorphic sections of the tangent bundle over very thick parts of \tec space. 

Part I begins with section \ref{s-coos} in which we estimate the formula of the Weil-Petersson curvature operator and restate the non-positivity of the Weil-Petersson curvature operator in terms of eigenvalues. Section \ref{s-ub} provides the upper bound Theorem~\ref{uubf-0} for the minimal eigenvalue. In section \ref{s-lb} we establish Theorem~\ref{lbfme-0}, the lower bound for the minimal eigenvalue in the thick part, and Thereom~\ref{ubfl^p}. We prove Theorem~\ref{largei-0} and Theorem~\ref{fftz-0} in section \ref{zpoev}. 

Part II spans two sections.  We begin in section~\ref{hmsc} with some estimates on harmonic Beltrami differentials, and then apply those estimates in section~\ref{sec:curv of hol lines} to prove Theorem~\ref{ubfhsc} and Corollary \ref{cor:bdd curv on seq}. \\

\begin{acknowledgement}
The authors would like to thank Maryam Mirzakhani, Hugo Parlier and Kasra Rafi for useful conversations. They also would like to thank Scott Wolpert for useful conversations and suggestions on part (2) of this article. They deeply thank Zheng (Zeno) Huang for bringing  Question \ref{qom} to their attention, his detailed comments for this article, and for his support over many years. The first author also gratefully acknowledges support
from the U.S.~National Science Foundation through grant DMS 1007383. This work was partially completed while the second author visited the Mathematical Sciences Center of Tsinghua university. Both authors acknowledge support from U.S. National Science Foundation grants DMS 1107452, 1107263, 1107367 ``RNMS: Geometric structures And Representation varieties"(the GEAR Network).
\end{acknowledgement}

\section{Notations and Preliminaries}\label{s-np} 

\subsection{Notation and Background on \wep curvatures.} \label{sec:wp background}
In this section, we set our notations and quickly review the relevant 
background material on the \wep metric and curvatures. We have two principal goals. For use in Part I, we prove Proposition~\ref{bfscalar}, which shows that the scalar curvatures are comparable to the $L^2$ norms of a pointwise Bergman sum of harmonic Beltrami differentials. For use in Part II, we prove Proposition~\ref{ratio}, which shows that the holomorphic sectional curvature along a harmonic Beltrami differential $\mu$ is estimated in terms of powers of its normalized $L^{\infty}$ norm.

To begin, recall that we denoted by $S_{g}$ a closed oriented surface of genus $g\geq 2$.  We may equip $S_g$ with a hyperbolic metric $\sigma(z)|dz|^2$, here written in a local conformal coordinate $z$ induced by the metric. Again, $\Te(S_g)$ is the Teichm\"uller space of surfaces of genus $g$, which we may construe as equivalence classes under the action of the group $\Diff_0$ of diffeomorphisms isotopic to the identity of the space of hyperbolic surfaces $X=(S_g,\sigma|dz|^2)$. The tangent space $T_X\Teich(S_g)$ at a point $X=(S_g,\sigma|dz|^2)$ is identified with the space of {\it harmonic Beltrami differentials} on $X$, i.e. forms on $X$ expressible as 
$\mu=\frac{\overline{\psi}}{\sigma}$ where $\psi \in H^0(X, K^2)$ is a holomorphic quadratic differential on $X$. Let $z=x+\textbf{i}y$ and $dA=\sigma(z)dxdy$ be the volume form. The \textit{Weil-Petersson metric} is the Hermitian
metric on $\Teich(S_g)$ arising from the the \textit{Petersson scalar  product}
\begin{equation}
 <\varphi,\psi>= \int_S \frac{\varphi \cdot \overline{\psi}}{\sigma^2}dA\nonumber
\end{equation}
via duality. We will concern ourselves primarily with its Riemannian part $g_{WP}$. Throughout this paper we denote the Teichm\"uller space endowed with the Weil-Petersson metric by $\Teich(S_g)$.\\ 

Set $D=-2(\Delta-2)^{-1}$ where $\Delta$ is the Beltrami-Laplace operator on $X=(S,\sigma|dz|^2)$. The operator $D$ is positive and self-adjoint. The following inequality follows from the maximum principle; see Lemma 5.1 in \cite{Wolf12} for details. 
\begin{lemma}[\cite{Wolf12}]\label{lowolf}
For any harmonic Beltrami differential $\mu$ on $X$, we have
\[D(|\mu|^2)\geq \frac{|\mu|^2}{3}.\]
\end{lemma}   

The following property is well-known to experts -- see for example Lemma~4.3 in \cite{LSY04}. For completeness, we include the proof.
\begin{lemma}\label{ufD} 
Let $D$ be the operator above. Then, for any complex-valued function $f\in C^{\infty}(X)$,
\[0\leq \int_{X}{(D(f)\overline{f})dA}\leq \int_{X}{|f|^2dA}.\]
\end{lemma}

\begin{proof}
Let $g=D(f)$, so that $f=\frac{-1}{2}(\Delta-2)g$. The left inequality follows directly by integrating by parts.

For the right inequality, begin by assuming that $f$ is real-valued. In that case decompose $f=\sum_{i=0}^{\infty}\phi_{i}$ as a linear combination of eigenfunctions of the Laplacian: here $\phi_{i}$ satisfies $\int_{X}(\phi_{i}\phi_j)dA=0$ for all $i\neq j$ and $\Delta \phi_{i}=\mu_{i}\phi_{i}$, where $\mu_{i}< 0$ is the $i$-th eigenvalue of  the Beltrami-Laplace operator $\Delta$. Since $\int_{X}(\phi_{i}\phi_j)dA=0$ for all $i\neq j$, we have
\begin{eqnarray*}
\int_{X}{(D(f)f)dA} &=& \int_{X}{(\sum_{i=0}^{\infty} \frac{2\phi_{i}}{2-\mu_{i}})(\sum_{i=0}^{\infty}\phi_{i})dA}\\
&=& \int_{X}{\sum_{i=0}^{\infty}\frac{2}{2-\mu_{i}}|\phi_{i}|^2dA}\\
&\leq&  \int_{X}{\sum_{i=0}^{\infty} |\phi_{i}|^2 dA}\\
&=&\int_{X}{f^2dA}.
\end{eqnarray*}
If $f$ is complex-valued, one applies the same method to the real and imaginary parts of $f$ separately, along with a standard use of the self-adjointness of $D$.
\end{proof}

We conclude this discussion of the operator $D=-2(\Delta-2)^{-1}$ by recording some elementary properties of its Green's function. 
\begin{proposition} \label{psfg}
Let $D$ be the operator $D=-2(\Delta-2)^{-1}$. Then there exists a Green function $G(w,z)$ for $D$ satisfying:\\
(1). $D(f)(z)=\int_{w \in X}{G(z,w)f(w)dA(w)}$ for any $f\in C^{\infty}((X,\sigma|dz|^2),\mathbb{C})$.\\
(2). $G(w,z)$ is positive.\\
(3). $G(w,z)$ is symmetric, i.e, $G(w,z)=G(z,w)$.
\end{proposition}

\begin{proof}
See for example \cite{Roelcke} and \cite{Wol86}.
\end{proof}

\subsection{The Riemannian tensor of the Weil-Petersson metric.} The curvature tensor of the \wep metric is given as follows. As described in the opening paragraph of section~\ref{sec:wp background}, let $\mu_{i},\mu_{j}$ be two elements in the tangent space $T_X\Teich(S_g)$ at $X$, so that the metric tensor might be written in local coordinates as
\begin{eqnarray*}
g_{i \overline{j}}=\int_X \mu_{i} \cdot  \overline{\mu_j} dA.
\end{eqnarray*} 

For the inverse of $(g_{i\overline{j}})$, we use the convention
\begin{eqnarray*}
g^{i\overline{j}} g_{k\overline{j}}=\delta_{ik}.
\end{eqnarray*}

Then the curvature tensor is given by
\begin{eqnarray*}
R_{i\overline{j}k\overline{l}}=\frac{\partial^2}{\partial t^{k}\partial \overline{t^{l}}}g_{i\overline{j}}-g^{s\overline{t}}\frac{\partial}{\partial t^{k}}g_{i\overline{t}}\frac{\partial}{\partial \overline{t^{l}}}g_{s\overline{j}}.
\end{eqnarray*}

The following curvature formula was established in \cite{Tromba86, Wol86}. One can also see \cite{Jost91} for a derivation from a third perspective. It has been applied to study various curvature properties of the Weil-Petersson metric. In \cite{Schumacher} Schumacher showed that $\Teich(S_g)$ has strongly negative curvature in the sense of Siu. Huang in his thesis \cite{Huang1} showed that the sectional curvatures of $\Teich(S_g)$ can not be bounded away from zero. Liu-Sun-Yau in \cite{LSY} showed that $\Teich(S_g)$ has dual Nakano negative curvature, which says that the complex curvature operator on the dual tangent bundle is positive in some sense. Motivated by the method in \cite{LSY}, the second author in \cite{Wu14-co} showed that the $\Teich(S_g)$ has negative semi-definite Riemannian curvature operator. One can also see \cite{ LSY04, LSY05, LSYY, Teo08, Wol08, Wolpert5} for other aspects of the curvature of $\Teich(S_g)$.

\begin{theorem}[Tromba, Wolpert]\label{cfow} 
The curvature tensor satisfies

\[R_{i\overline{j}k\overline{l}}=\int_{X} D(\mu_{i}\mu_{\overline{j}})\cdot (\mu_{k}\mu_{\overline{l}}) dA+\int_{X} D(\mu_{i}\mu_{\overline{l}})\cdot (\mu_{k}\mu_{\overline{j}}) dA.\]

\end{theorem}

\subsubsection{\wep Scalar Curvature.}
Let $\{\mu_{i}\}_{i=1}^{3g-3}$ be a holomorphic orthonormal basis representing the tangent space $T_X\Teich(S_g)$ at $X$. Then the Ricci curvature of $\Teich(S_g)$ at $X$ in the direction $\mu_{i}$ is given by
\begin{eqnarray*}
&&Ric(\mu_{i})=-\sum_{j=1}^{3g-3}R_{i\overline{j}j\overline{i}}\\
&=&-\sum_{j=1}^{3g-3}(\int_{X} D(\mu_{i}\mu_{\overline{j}})\cdot (\mu_{j}\mu_{\overline{i}}) dA+\int_{X} D(|\mu_{i}|^2)\cdot (|\mu_{j}|^2) dA).
\end{eqnarray*}

Since the scalar curvature $\Sca(X)$ at $X$ is the trace of the Ricci tensor, we find we may express the scalar curvature as  
\begin{equation} \label{eqn:scalar curvature} 
\Sca(X)=-\sum_{i=1}^{3g-3}\sum_{j=1}^{3g-3}(\int_{X} D(\mu_{i}\mu_{\overline{j}})\cdot (\mu_{j}\mu_{\overline{i}}) dA+\int_{X} D(|\mu_{i}|^2)\cdot (|\mu_{j}|^2) dA).
\end{equation}

From this expression for the scalar curvature, we may apply our estimates for the operator $D$ in section~\ref{sec:wp background} to obtain estimates for the scalar curvature from above and below in terms of pointwise-defined quantities.
\begin{proposition}\label{bfscalar}
For any $X \in \Teich(S_g)$, the scalar curvature $\Sca(X)$ at $X$ satisfies 

\begin{equation*}
 -2\int_{X}{(\sum_{i=1}^{3g-3}|\mu_{i}|^2)^2}dA \leq \Sca(X)\leq -\frac{1}{3}\int_{X}{(\sum_{i=1}^{3g-3}|\mu_{i}|^2)^2dA}
\end{equation*}
where $\{\mu_{i}\}_{i=1}^{3g-3}$ is any holomorphic orthonormal basis of the tangent space at $X$.
\end{proposition}

\begin{proof}
We begin with the right-hand-side inequality. By Lemma \ref{lowolf}, the first term in the expression \eqref{eqn:scalar curvature} has a sign, so that we can then apply Lemma \ref{ufD} to find
\begin{align*}
-\Sca(X)&=\sum_{i=1}^{3g-3}\sum_{j=1}^{3g-3}(\int_{X} D(\mu_{i}\mu_{\overline{j}})\cdot (\mu_{j}\mu_{\overline{i}}) dA+\int_{X} D(|\mu_{i}|^2)\cdot (|\mu_{j}|^2) dA)\\
&\geq\sum_{i=1}^{3g-3}\sum_{j=1}^{3g-3}\int_{X} D(|\mu_{i}|^2)\cdot (|\mu_{j}|^2) dA \\
&\geq \frac{1}{3}\int_{X}{(\sum_{i=1}^{3g-3}|\mu_{i}|^2)^2}dA.
\end{align*}

For the left hand side inequality, we use the right side of the estimate in Lemma \ref{ufD} twice, with the preliminary step of bringing the sum into the integrands. We compute 
\begin{eqnarray*}
-\Sca(X)&=&\sum_{i=1}^{3g-3}\sum_{j=1}^{3g-3}(\int_{X} D(\mu_{i}\mu_{\overline{j}})\cdot (\mu_{j}\mu_{\overline{i}}) dA+\int_{X} D(|\mu_{i}|^2)\cdot (|\mu_{j}|^2) dA)\\
&=&\sum_{i=1}^{3g-3}\sum_{j=1}^{3g-3}\int_{X} D(\mu_{i}\mu_{\overline{j}})\cdot (\mu_{j}\mu_{\overline{i}}) dA+\int_{X} \sum_{i=1}^{3g-3}\sum_{j=1}^{3g-3}D(|\mu_{i}|^2)\cdot (|\mu_{j}|^2) dA\\
&\leq&\sum_{i=1}^{3g-3}\sum_{j=1}^{3g-3}\int_{X} (\mu_{i}\mu_{\overline{j}})\cdot (\mu_{j}\mu_{\overline{i}}) dA+\int_{X}D(\sum_{i=1}^{3g-3}|\mu_{i}|^2)\cdot (\sum_{j=1}^{3g-3}|\mu_{j}|^2) dA\\
&=&\sum_{i=1}^{3g-3}\sum_{j=1}^{3g-3}\int_{X} |\mu_{i}|^2 |\mu_{j}|^2 dA+\int_{X} D(\sum_{i=1}^{3g-3}|\mu_{i}|^2)\cdot (\sum_{i=1}^{3g-3}|\mu_{i}|^2) dA\\
&\leq & 2\int_{X}{(\sum_{i=1}^{3g-3}|\mu_{i}|^2)^2dA}.
\end{eqnarray*}
\end{proof}

\begin{remark}\label{ubfs}
The Cauchy-Schwarz inequality applied to the right hand side estimate in Proposition~\ref{bfscalar} leads to a numerical lower bound for $-\Sca(X)$ in the following way:
\begin{eqnarray*}
-\Sca(X)&\geq& \frac{\int_{X}{(\sum_{i=1}^{3g-3}|\mu_{i}|^2)^2}dA}{3}\\
&\geq& \frac{(\int_{X}{(\sum_{i=1}^{3g-3}|\mu_{i}|^2)}dA)^2}{3Area(S)}=\frac{3(g-1)}{4\pi}.
\end{eqnarray*} 
Now, without using Lemma \ref{lowolf}, Wolpert in \cite{Wol86} proved a better lower bound as $-\Sca(X)\geq \frac{3(3g-2)}{4\pi}$ by expanding $\sum_{i=1}^{3g-3}|\mu_{i}|^2$ (and $D(\sum_{i=1}^{3g-3}|\mu_{i}|^2)$) in the definition \eqref{eqn:scalar curvature} in terms of eigenfunctions of $\Delta$ at $X$. On the other hand, our goal in Proposition \ref{bfscalar} is to estimate $-\Sca(X)$ is estimated by the integral $\int_{X}{(\sum_{i=1}^{3g-3}|\mu_{i}|^2)^2}dA$, an expression we will apply later, so we content ourselves with the bound above. 
\end{remark}

\subsubsection{Weil-Petersson holomorphic sectional curvatures.}\label{subsec:hol-sec-curv}
Recall the holomorphic sectional curvature is a sectional curvature along a holomorphic line. Let $\mu \in T_X\Teich(S_g)$. Then Theorem \ref{cfow} tells that the holomorphic sectional curvature $K(\mu)$ along the holomorphic line spanned by $\mu$ is

\[K(\mu)=\frac{-2\cdot \int_{X} D(|\mu|^2)\cdot (|\mu|^2) dA}{||\mu||_{WP}^4}.\]

From Lemma \ref{lowolf}, Lemma \ref{ufD} and the equation above we have
\begin{proposition}\label{l4}
For any $\mu \in T_X\Teich(S_g)$, the holomorphic sectional curvature $K(\mu)$ satisfies
\[ -\frac{2\cdot \int_{X} |\mu|^4 dA}{||\mu||_{WP}^4}\leq K(\mu)\leq -\frac{2\cdot \int_{X} |\mu|^4 dA}{3||\mu||_{WP}^4}.\]
\end{proposition}

We will refer several times to the a constant $C(\inj)$ depending on the injectivity radius, which we pause now to define.

\begin{definition} \label{defn:C-inj}
Set $C(\inj(X))=(\frac{4\pi}{3}(1-(\frac{4e^{\inj(X)}}{(1+e^{\inj(X)})^2})^3))^{-1}$.
\end{definition}

\begin{remark}
We note here that as the injectivity radius $\inj(X)$ tends to zero, the constant $C(\inj(X)) = \frac{1}{\pi\inj(X)^2} + o(\frac{1}{\inj(X)^2})$. And $C(\inj(X))$ tends to $\frac{3}{4\pi}$ as $\inj(X)$ goes to infinity.
\end{remark}

Next we recall the following proposition which is implicitly proved in \cite{Huang, Teo08, Wolpert5}. For the sake of keeping the exposition self-contained, we give the outline of the proof which follows the identical argument as that of Proposition 3.1 in \cite{Teo08}.

\begin{proposition}\label{lnubw}
Let $(X,\sigma|dz|^2)$ be a closed hyperbolic surface and $\mu \in T_{X}\Teich(S_g)$ be a harmonic Beltrami differential of $X$. Then, for any $q\in X$ we have
\[|\mu(q)|^2 \leq C(r) \int_{B(q;r)}{|\mu(z)|^2 dA(z)} \quad \forall \ 0<r\leq \inj(q). \]
Here $\inj(q)$ is the injectivity radius of $q$ in $X$ and the explicit constant $C(r)$ is given in Definition~\ref{defn:C-inj}.
\end{proposition}

\begin{proof}
Let $\mu \in T_{X}\Teich(S_g)$; we denote its lift into the hyperbolic disk by $\nu$. For any point $q \in X$, we may conjugate the lift by a linear fractional transformation so that 
$$\nu(0)=\mu(q).$$
For any $r \in (0, \inj(q)]$, elementary hyperbolic geometry gives that the hyperbolic disk of radius $r$ centered at $0$ is 
\[B(0;r)=\{z\in \D; |z|\leq \frac{e^r-1}{e^r+1}\}.\] 
Since $\nu$ is a harmonic Beltrami differential on the disk, there exists a holomorphic function $f$ such that 
\[\nu(z)=\frac{\overline{f(z)}}{\frac{4}{(1-|z|^2)^2}}=\frac{\overline{f(z)}(1-|z|^2)^2}{4}.\]
Of course, since $f$ is holomorphic, we may expand it as
\[f(z)=\sum_{n=0}^{\infty}a_{n}z^n\]
where $a_n$ are complex numbers. In particular we have
\[|\mu(q)|=|\nu(0)|=\frac{|a_0|}{4}.\]

Therefore, for any $r \in (0, \inj(q)]$, we have
\begin{eqnarray*}
\int_{B(q;r)}{|\mu(z)|^2 dA(z)}&=&\int_{B(0;r)}{|\nu(z)|^2 dA(z)}\\
&=& \int_{B(0;r)}{|f(z)|^2 \frac{(1-|z|^2)^2}{4}|dz|^2 }\\
&=& \int_{0}^{\frac{e^r-1}{e^r+1}}\int_{0}^{2\pi}|\sum_{n=0}^{\infty}a_{n}u^n e^{\textbf{i}n\theta}|^2 \frac{(1-u^2)^2}{4}udu d\theta\\
&=&\frac{\pi}{2}\int_{0}^{\frac{e^r-1}{e^r+1}}\sum_{n=0}^{\infty}|a_{n}|^2 u^{2n}(1-u^2)^2 udu\\
&\geq& \frac{\pi}{2}|a_0|^2 \int_{0}^{\frac{e^r-1}{e^r+1}}(1-u^2)^2 udu\\
&=&\frac{4\pi}{3}(1-(\frac{4e^{r}}{(1+e^{r})^2})^3) |\mu(q)|^2. 
\end{eqnarray*}
The conclusion follows by dividing the constant on both sides.
\end{proof}

We combine these last two propositions to obtain the following proposition which is crucial in the proof of Theorem \ref{ubfhsc} (which produces a holomorphic section whose curvature is pinched): this proposition traps the holomorphic sectional curvature between 
powers of $L^{\infty}$ bounds on (normalized) holomorphic Beltrami differentials.

\begin{proposition}\label{ratio}
Let $X \in \Teich(S_g)$ with $\inj(X)\geq 1$. Then for any $\mu \in T_X\Teich(S_g)$, there exists a constant $C_0>0$, independent of the genus, such that the holomorphic sectional curvature $K(\mu)$ satisfies
\[  -\frac{2\cdot \sup_{z\in X}|\mu(z)|^2}{||\mu||_{WP}^2}\leq K(\mu)\leq -\frac{C_0\cdot \sup_{z\in X}|\mu(z)|^4}{||\mu||_{WP}^4}.\]
\end{proposition}
\begin{proof}
We begin with the left hand side. By Proposition \ref{l4} we have
\[ -K(\mu)\leq \frac{2\cdot  \sup_{z\in X}|\mu(z)|^2 \cdot \int_{X} |\mu|^2 dA}{||\mu||_{WP}^4}= \frac{2\cdot \sup_{z\in X}|\mu(z)|^2}{||\mu||_{WP}^2}.\]

For the right hand side, since $X$ is compact, we let $p \in X$ such that $|\mu(p)|= \sup_{z\in X}|\mu(z)|$. Let $B(p;1)$ be the closed ball of hyperbolic radius 1 centered at $p$. Use the Cauchy-Schwarz inequality to obtain
\begin{eqnarray*}
\int_{X} |\mu|^4 dA &\geq& \int_{B(p;1)} (|\mu|^4) dA\geq \frac{(\int_{B(p;1)} |\mu|^2 dA)^2}{Vol(B(p;1))}.
\end{eqnarray*}
Applying Proposition \ref{lnubw} to the last term, we find
\begin{eqnarray*}
\int_{X} |\mu|^4 dA \geq \frac{ |\mu(p)|^4}{C(1)^2\cdot Vol(B(p;1))}=\frac{ \sup_{z\in X}|\mu(z)|^4}{C(1)^2\cdot Vol(B(p;1))}.
\end{eqnarray*}

Finally, by choosing $C_0=\frac{2}{3C(1)^2\cdot Vol(B(p;1))}$ (where C(1) is the explicit constant from Definition~\ref{defn:C-inj}), the conclusion follows from Proposition \ref{l4}.
\end{proof}

\begin{remark}
Combining Proposition \ref{lnubw} and Proposition \ref{ratio}, one may conclude that the holomorphic sectional curvatures of the $\epsilon_0$-thick part of the moduli space are uniformly bounded below by a constant only depending on $\epsilon_0$, which was proved in \cite{Huang}. 
\end{remark}

\begin{remark}
If one would like to search for harmonic Beltrami differentials with uniformly negative holomorphic sectional curvatures as the genus goes to infinity, Proposition \ref{ratio} tells that it suffices to find  such differentials with Weil-Petersson $L^2$ norms bounded above and \wep $L ^\infty$ norms bounded below.  We develop this theme in the second part of this paper.
\end{remark}

\subsection{Sufficiently thick surfaces.} \label{sec:suff thick}
Often in \tec theory, one studies problems in the setting where some curve $\alpha$ on the surface $S$ has hyperbolic $X$-length $\ell_X(\alpha)$ small, say 
$\ell_X(\alpha) \leq \epsilon_0$. Alternatively, also often studies issues where the surface has injectivity radius $\inj(X)$ bounded away from zero, say $\inj(X) \geq \epsilon_1$. In this paper, a number of our results are in the region where the hyperbolic surface $X$ has a large injectivity radius, where 'large' here means large enough for some {\it a priori} bounds to apply.

\begin{definition}\label{defn: suff thick}
For a given large constant $C>0$, a surface $X$ is {\it $C$-sufficiently thick} or {\it sufficiently thick} (if the constant $C$ is clear from the context) if $\inj(X) \ge C$.
\end{definition}

We remark that there often is a non-empty subspace $\mathbb{M}^{\ge C}$ of sufficiently thick surfaces of the moduli space $\mathbb{M}$. It is well-known that there exist sequences of hyperbolic surfaces whose injectivity radii grow without bounds. For example, the fundamental group of a closed hyperbolic surface $X$ is residually finite (one can see \cite{Hempel72} for details). After taking finite normal covers to remove simple closed curves of bounded lengths in $X$,
we can find a sequence of hyperbolic surfaces $X_{g_k}$ with injectivity radii $\inj(X_{g_k}) \to \infty$ as $g_k \to \infty$. Moreover, Buser and Sarnak \cite{BS94} proved  that there exists a family of closed surfaces $X_{k}$ of genus $g_k$ with $g_k \to \infty$ as $k \to \infty$ and $\inj(X_k)\geq \frac{2\ln{g_k}}{3}$.

\part{Bounds on the \wep curvature operator}

\section{The Weil-Petersson curvature operator at $X$ and its eigenvalues}\label{s-coos}

In this section, we begin our study of the \wep curvature operator $\tilde{Q}$, establishing some preliminary upper and lower bounds for $<\tilde{Q}(A), A>$ in Proposition~\ref{upfco}. Those estimates rely on a formula for $<\tilde{Q}(A), A>$ displayed in Proposition~\ref{cof} which is particularly well-adapted for estimations.  The first subsection provides notation and context.

Before we study the curvature operator on $\Teich(S_g)$, we set some notation. Let $U \subset \Teich(S_g)$ be a neighborhood of $X$ in \tec space, and let $(t_{1},t_{2},\cdots,t_{3g-3})$ be a system of local holomorphic coordinates on $U$ such that $\{t_{i}(X)=\mu_{i}\}_{1\leq i \leq 3g-3}$ is holomorphic orthonormal at $X$ with $t_{i}=x_{i}+\textbf{i}y_{i} \ \ (1\leq i \leq 3g-3)$. Then $(x_{1},x_{2},\cdots,x_{3g-3},y_{1},y_{2},\cdots,y_{3g-3})$ are real smooth coordinates in $U$ which relate to the complex coordinates as
\begin{eqnarray*}
\frac{\partial}{\partial x_{i}} =\frac{\partial}{\partial t_{i}}+\frac{\partial}{\partial \overline{t_{i}}}, \ \ \frac{\partial}{\partial y_{i}} =\textbf{i}(\frac{\partial}{\partial t_{i}}-\frac{\partial}{\partial \overline{t_{i}}}).
\end{eqnarray*}

Let $T\Teich(S_g)$ be the real tangent bundle of $\Teich(S_g)$ and $\wedge^{2}T\Teich(S_g)$ be the exterior wedge product of $T\Teich(S_g)$ with itself. For any $X \in U$, we have

\[T_{X}\Teich(S_g)=\Span\{\frac{\partial}{\partial x_{i}}(X),\frac{\partial}{\partial y_{j}}(X)\}_{1\leq i,j\leq3g-3}.\] 

and

\[\wedge^{2}T\Teich(S_g)=\Span\{\frac{\partial}{\partial x_{i}}\wedge \frac{\partial}{\partial x_{j}}, 
\frac{\partial}{\partial x_{k}}\wedge \frac{\partial}{\partial y_{l}}, \frac{\partial}{\partial y_{m}}\wedge \frac{\partial}{\partial y_{n}}\}.\]

The space $\wedge^{2}T\Teich(S_g)$ has (real) dimension $(3g-3)(6g-7)$. Let $<,>$ refer to the pairing of vectors with respect to  the Weil-Petersson metric. The natural inner product on $\wedge^{2}T\Teich(S_g)$, associated to the Weil-Petersson metric, is given by
\begin{equation}\label{eqn: mowedge}
<V_1\wedge V_2, V_3\wedge V_4>_{eu}:=<V_1,V_3><V_2,V_4>-<V_1,V_4><V_2,V_3>
\end{equation}
and extended linearly, where $V_i$ are real vectors. One may refer to (\cite{Slang}, p. 238) for more details.

In terms of these real coordinates, the Weil-Petersson curvature operator $\tilde{Q}$ may be described in the following way. Let $X\in \Teich(S_g)$, let $R$ be the Riemannian curvature tensor of the Weil-Petersson metric, and let $\{e_{i}\}_{i=1}^{6g-6}$ be an orthonormal basis of $T_{X}\Teich(S_g)$. Set
\[R_{ijkl}:=<R(e_{i},e_{j})e_{k},e_{l}>.\]

It is clear that
\[\wedge^{2}T_X\Teich(S_g)=\Span\{e_{i}\wedge e_{j}\}_{1 \leq i<j \leq (6g-6)}.\]
Then the operator $\tilde{Q}:\wedge^{2}T_X\Teich(S_g) \rightarrow \wedge^{2}T_X\Teich(S_g)$ may be written in this notation as
\[\tilde{Q}(\sum_{1\leq i<j\leq (6g-6)}a_{ij}e_{i}\wedge e_{j}):=\sum_{1\leq i<j\leq (6g-6)}\sum_{1\leq k<l\leq (6g-6)}a_{ij}R_{ijkl}e_{k}\wedge e_{l},\]
where the coefficients $a_{ij}$ are set to be real.

 From equation (\ref{eqn: mowedge}) it is easy to see that $\{e_{i}\wedge e_{j}\}_{1 \leq i<j \leq (6g-6)}$ is an orthonormal basis of $\wedge^{2}T_X\Teich(S_g)$. Then,
\begin{eqnarray*}
&&<\tilde{Q}(\sum_{1\leq i<j\leq (6g-6)}a_{ij}e_{i}\wedge e_{j}),\sum_{1\leq i<j\leq (6g-6)}b_{ij}e_{i}\wedge e_{j}>_{eu}\\
&=&\sum_{1\leq i<j\leq (6g-6)}\sum_{1\leq k<l\leq (6g-6)}a_{ij}b_{kl}R_{ijkl},
\end{eqnarray*}
where here again the coefficients $b_{ij}$ are real. 

We define the \textsl{associated (bilinear, symmetric) curvature form} to be the bilinear form $Q$ on $\wedge^{2}T\Teich(S_g)$ given by
$Q(V_{1}\wedge V_{2},V_{3}\wedge V_{4})=R(V_{1},V_{2},V_{3},V_{4})$ and extended linearly, where the $V_{i}$ are real vectors. It is easy to see that $Q$ is a bilinear symmetric form (one can see more details in \cite{Jost}). 

In this notation, we see that we may write $Q(A,A)$ as $Q(A,A)=<\tilde{Q}(A),A>_{eu}$ for all $A\in \wedge^{2}T\Teich(S_g)$. By the symmetry of the Riemannian curvature tensor and the definition of the scalar curvature, we then find
\begin{lemma}\label{bfoco}
(1). $\tilde{Q}$ is self-adjoint.\\
(2). The trace $Tr(\tilde{Q})(X)$ of $Q$ at $X$ satisfies $Tr(\tilde{Q})(X)=\Sca(X)$.
\end{lemma}
For more details on the Riemannian curvature operator, one can see section 2.2 in \cite{Pet06}. 

\subsection{The eigenvalues of the Weil-Petersson curvature operator}
The action of the almost complex structure $\textbf{J}$ on $T_{X}\Teich(S_g)$ extends to a natural action of $\textbf{J}$ on $\wedge^{2}T_{X}\Teich(S_g)$, defined as follows on a basis 
\begin{eqnarray*}
\begin{cases}
\textbf{J}\circ\frac{\partial}{\partial x_{i}}\wedge \frac{\partial}{\partial x_{j}}:=\frac{\partial}{\partial y_{i}}\wedge \frac{\partial}{\partial y_{j}},\\
\textbf{J}\circ\frac{\partial}{\partial x_{i}}\wedge \frac{\partial}{\partial y_{j}}:=-\frac{\partial}{\partial y_{i}}\wedge \frac{\partial}{\partial x_{j}}=\frac{\partial}{\partial x_{j}}\wedge \frac{\partial}{\partial y_{i}},\\
\textbf{J}\circ\frac{\partial}{\partial y_{i}}\wedge \frac{\partial}{\partial y_{j}}:=\frac{\partial}{\partial x_{i}}\wedge \frac{\partial}{\partial x_{j}},
\end{cases}
\end{eqnarray*}
and then extended linearly. It is easy to see that, as an operator on $\wedge^{2}T_{X}\Teich(S_g)$, we have $\textbf{J}\circ\textbf{J}=id.$

Part of the second author's thesis  shows that for any $X\in \Teich(S_g)$, the curvature operator $\tilde{Q}$ at $X$ is negative semi-definite. More precisely,

\begin{theorem}[\cite{Wu14-co}]\label{conp}
Let $S=S_{g}$ be a closed surface of genus $g>1$ and $\Teich(S_g)$ be the Teichm\"uller space of $S$ endowed with the Weil-Petersson metric. And let $\textbf{J}$ be the almost complex structure on $\Teich(S_g)$ and $Q$ be the associated curvature form on $\Teich(S_g)$. Then, for any $X \in \Teich(S_g)$, we have\\
(1). $\tilde{Q}$ is negative semi-definite, i.e., $<\tilde{Q}(A), A>_{eu}=Q(A,A)\leq 0$ for all $A \in \wedge^{2}T_{X}\Teich(S_g)$.\\
(2). $Q(A,A)=0$ if and only if there exists an element $B$ in $\wedge^{2}T_{X}\Teich(S_g)$ such that
$A=B-\textbf{J}\circ B$.
\end{theorem}

Recall that $\lambda$ is called an \textsl{eigenvalue} of $\tilde{Q}$ if there exists an element $A \in \wedge^{2}T_{X}\Teich(S_g)$ such that $\tilde{Q}(A)=\lambda \cdot A$. In that case, the element $A$ is called the \textsl{eigenvector} associated to $\lambda$.   Since $\tilde{Q}$ is self-adjoint, all of the eigenvalues of $\tilde{Q}$ are real. Since $\Dim(\Teich(S_g))=6g-6$, we see from Theorem \ref{conp} that there exist $(3g-3)(6g-7)$ non-positive eigenvalues $\{\nu_{i}\}_{i=1}^{(3g-3)(6g-7)}$ of $\tilde{Q}$ where we set $\nu_{i+1}\leq \nu_{i}\leq 0$ for all $0\leq i \leq ((3g-3)(6g-7)-1)$. Since the Weil-Petersson sectional curvature is negative (see \cite{Wol86}), not all $\{\nu_{i}\}$ vanish. Our focus is on the non-zero eigenvalues. Refining the analysis of Theorem \ref{conp}, we prove that the number of non-zero eigenvalues of the \wep curvature operator $\tilde{Q}$ on \tec space $\Teich(S_g)$ is constant.

\begin{theorem}\label{(3g-3)^2}
For any $X \in \Teich(S_g)$, the Weil-Petersson curvature operator $\tilde{Q}$ at $X$ has exactly $(3g-3)^2$ negative eigenvalues.
\end{theorem}

To prepare for the proof of Theorem \ref{(3g-3)^2}, we note the following standard lemma.

\begin{lemma}\label{zoco}
For any $B \in \wedge^{2}T_{X}\Teich(S_g)$, the element $A$ defined by $A=B-\textbf{J}\circ B$ is a $0$-eigenvector, i.e. $\tilde{Q}(A)=0$. 
\end{lemma}
\begin{proof}
Of course, by part (2) of Theorem \ref{conp} we have $Q(A,A)=0$. Now let $C\in \wedge^{2}T_{X}\Teich(S_g)$. By part (1) of Theorem \ref{conp}, for all $t\in \mathbb{R}$, we have 
\begin{eqnarray*}
0 &\geq& Q(C+tA,C+tA)=<Q(C+tA),C+tA>_{eu}\\
&=& <\tilde{Q}(C),C>_{eu}+2t<\tilde{Q}(A),C>_{eu}.
\end{eqnarray*}
Since $t$ is arbitrary, we see that we could choose $t$ to make the expression $<\tilde{Q}(C),C>_{eu}+2 <\tilde{Q}(A),C>_{eu} t$  above positive unless $<\tilde{Q}(A),C>_{eu}=0$. The conclusion then follows by choosing $C=\tilde{Q}(A)$ so that we have $0 = <\tilde{Q}(A),C>_{eu}= <\tilde{Q}(A),\tilde{Q}(A)>_{eu}$, yielding $\tilde{Q}(A) = 0$.
\end{proof}

\begin{proof}[Proof of Theorem \ref{(3g-3)^2}] It suffices to show that the dimension of the space of zero-eigenvectors of $\tilde{Q}$ is equal to $(3g-3)(3g-4)$. First, it is clear that the elements in $\{\frac{\partial}{\partial x_{i}}\wedge \frac{\partial}{\partial x_{j}}-\frac{\partial}{\partial y_{i}}\wedge \frac{\partial}{\partial y_{j}}\}_{1\leq i<j\leq (3g-3)}, \{\frac{\partial}{\partial x_{i}}\wedge \frac{\partial}{\partial y_{j}}-\frac{\partial}{\partial x_{j}}\wedge \frac{\partial}{\partial y_{i}}\}_{1\leq i<j\leq (3g-3)}$ are linearly independent in $\wedge^{2}T_{X}\Teich(S_g)$. By Lemma \ref{zoco}, we see that 
$\Span\{\{\frac{\partial}{\partial x_{i}}\wedge \frac{\partial}{\partial x_{j}}-\frac{\partial}{\partial y_{i}}\wedge \frac{\partial}{\partial y_{j}}\}_{1\leq i<j\leq (3g-3)}, \{\frac{\partial}{\partial x_{i}}\wedge \frac{\partial}{\partial y_{j}}-\frac{\partial}{\partial x_{j}}\wedge \frac{\partial}{\partial y_{i}}\}_{1\leq i<j\leq (3g-3)}\}$ is contained in the space of zero-eigenvectors of $\tilde{Q}$. Thus, \[\Dim\{A\in \wedge^{2}T_{X}\Teich(S_g);\ \ \tilde{Q}(A)=0\}\geq 2\cdot (3g-3)(3g-4)/2=(3g-3)(3g-4).\]
On the other hand, let $A=\sum_{ij}(a_{ij}\frac{\partial}{\partial x_{i}}\wedge \frac{\partial}{\partial x_{j}}+b_{ij}\frac{\partial}{\partial x_{i}}\wedge \frac{\partial}{\partial y_{j}}+ c_{ij}\frac{\partial}{\partial y_{i}}\wedge \frac{\partial}{\partial y_{j}})\in \wedge^{2}T_{X}\Teich(S_g)$ with $\tilde{Q}(A)=0$. By part (2) of Theorem \ref{conp}, 
we know that
$A=B-\textbf{J}\circ B$ for some $B \in \wedge^{2}T_{X}\Teich(S_g)$. Then we have, for all $1\leq i, j \leq (3g-3)$,
\[a_{ij}+c_{ij}=a_{ji}+c_{ji}, \ \ b_{ij}+b_{ji}=0.\]
We rewrite it as 
\[a_{ij}-a_{ji}=-(c_{ij}-c_{ji}),\ \  b_{ij}=-b_{ji}.\]
Thus,
\[A=\sum_{i<j}(a_{ij}-a_{ji})(\frac{\partial}{\partial x_{i}}\wedge \frac{\partial}{\partial x_{j}}-\frac{\partial}{\partial y_{i}}\wedge \frac{\partial}{\partial y_{j}})+\sum_{i<j}b_{ij}(\frac{\partial}{\partial x_{i}}\wedge \frac{\partial}{\partial y_{j}}-\frac{\partial}{\partial x_{j}}\wedge \frac{\partial}{\partial y_{i}}).\]
In particular $A \in Span\{\{\frac{\partial}{\partial x_{i}}\wedge \frac{\partial}{\partial x_{j}}-\frac{\partial}{\partial y_{i}}\wedge \frac{\partial}{\partial y_{j}}\}_{1\leq i<j\leq (3g-3)}, \{\frac{\partial}{\partial x_{i}}\wedge \frac{\partial}{\partial y_{j}}-\frac{\partial}{\partial x_{j}}\wedge \frac{\partial}{\partial y_{i}}\}_{1\leq i<j\leq (3g-3)}\}$. Since $A$ is arbitrary with $\tilde{Q}(A)=0$,
\[\dim\{A\in \wedge^{2}T_{X}\Teich(S_g);\ \ \tilde{Q}(A)=0\}\leq 2\cdot (3g-3)(3g-4)/2=(3g-3)(3g-4).\] 
Therefore, 
\[\dim\{A\in \wedge^{2}T_{X}\Teich(S_g);\ \ \tilde{Q}(A)=0\}=(3g-3)(3g-4).\]
\end{proof}

From Theorem \ref{conp} and Theorem \ref{(3g-3)^2}, we know that for any $X \in \Teich(S_g)$, the Weil-Petersson curvature operator $\tilde{Q}$ at $X$ has $(3g-3)^2$ negative eigenvalues. We denote them by 
\[\lambda_{(3g-3)^2}(X)\leq \lambda_{(3g-3)^2-1}(X) \leq \lambda_{(3g-3)^2-2}(X)\leq \cdots \lambda_{2}(X)\leq \lambda_{1}(X)<0\]
We close this subsection by rewriting the smallest eigenvalue $\lambda_{(3g-3)^2}(X)$ as follows
\begin{eqnarray}\label{meve}
\lambda_{\min}(X):=\lambda_{(3g-3)^2}(X)=\min_{A\in \wedge^2T_{X}\Teich(S_g), ||A||_{eu}=1}Q(A,A).
\end{eqnarray}

\subsection{The Weil-Petersson curvature operator formula}

In this section, we recall a formula from \cite{Wu14-co} for the curvature operator applied to an element $A \in \wedge^{2}T_{X}\Teich(S_g)$ of $\wedge^{2}T_{X}\Teich(S_g)$. This formula is particularly well-suited to estimating the associated curvature form $Q(A,A)$, and so from the formula we derive some estimates that we will need in the later sections.

To begin, we express an arbitrary element $A \in \wedge^{2}T_{X}\Teich(S_g)$ in coordinates as

\begin{equation} \label{eqn:defnA}
A=\sum_{ij}(a_{ij}\frac{\partial}{\partial x_{i}}\wedge \frac{\partial}{\partial x_{j}}+b_{ij}\frac{\partial}{\partial x_{i}}\wedge \frac{\partial}{\partial y_{j}}+ c_{ij}\frac{\partial}{\partial y_{i}}\wedge \frac{\partial}{\partial y_{j}})\in \wedge^{2}T_{X}\Teich(S_g)
\end{equation}
where $a_{ij}, b_{ij},c_{ij}$ are real numbers. The following formula is crucial not only for our analysis in this section, but indeed for much of our work in this part of the paper.
\begin{proposition} \label{cof}
With $A$ defined as above, we may write $Q(A,A)$ as

\begin{small}
\begin{align} \label{eqn:Q(A,A)decomposition}
Q(A,A)&=-4\int_{X}{D (\Im\{F(z,z)+\textbf{i}H(z,z)\})\cdot(\Im\{F(z,z)+\textbf{i}H(z,z)\})dA(z)}\\ \notag
&-2\int_{X\times X}G(z,w)|F(z,w)+\textbf{i}H(z,w))|^2dA(w)dA(z)\\ \notag
&+2\Re\{\int_{X\times X}G(z,w)(F(z,w)+\textbf{i}H(z,w))(F(w,z)+\textbf{i}H(w,z))dA(w)dA(z)\} \notag
\end{align}
where $G(z,w)$ is the Green's function for the operator $D$, the expression $F(z,w)$ is set to be  $F(z,w)=\sum_{i,j=1}^{3g-3}{(a_{ij}+c_{ij})\mu_{i}(w)\cdot \overline{\mu_{j}(z)}}$ and $H(z,w)$ is defined as 
$H(z,w)=\sum_{i,j=1}^{3g-3}{b_{ij}\mu_{i}(w)\cdot \overline{\mu_{j}(z)}}$.
\end{small}

\end{proposition} 

\begin{proof}
The conclusion directly follows from Proposition 4.1 and Proposition 4.3 in \cite{Wu14-co}.
\end{proof}

Using this decomposition of $Q(A,A)$, we estimate its value in the following proposition.
The proof uses a similar idea as in the proof of Theorem 4.4 in \cite{Wu14-co}.
\begin{proposition}\label{upfco}
Under the same conditions as in Proposition \ref{cof}, we have\\
$(1).$ $-Q(A,A) \geq 4\int_{X}{D (\Im\{F(z,z)+\textbf{i}H(z,z)\})\cdot(\Im\{F(z,z)+\textbf{i}H(z,z)\})dA(z)}$\\
$(2).$ $-Q(A,A) \leq 8\cdot\int_{X}(|F(z,z)|^2+ |H(z,z)|^2)dA(z)+8\int_{X\times X}G(z,w)(|F(z,w)|^2+|H(z,w)|^2)dA(w)dA(z).$ 
\end{proposition}

\begin{proof}
 We begin with the third term in equation (\ref{eqn:Q(A,A)decomposition}). By the Cauchy-Schwarz inequality, we may write
\begin{small}
\begin{align}\label{3.2}
&|\int_{X\times X}G(z,w)(F(z,w)+\textbf{i}H(z,w))(F(w,z)+\textbf{i}H(w,z))dA(w)dA(z)| \\ \notag
&\leq \int_{X\times X}G(z,w)|(F(z,w)+\textbf{i}H(z,w))(F(w,z)+\textbf{i}H(w,z))|dA(w)dA(z)\nonumber\\ \notag
&\leq \sqrt{\int_{X\times X}G(z,w)|(F(z,w)+\textbf{i}H(z,w))|^2dA(w)dA(z)} \nonumber\\ \notag
&\times \sqrt{\int_{X\times X}G(z,w)|(F(w,z)+\textbf{i}H(w,z))|^2dA(w)dA(z)}\nonumber\\ \notag
&= \int_{X\times X}G(z,w)|(F(z,w)+\textbf{i}H(z,w))|^2dA(w)dA(z).\notag
\end{align}
\end{small}
The last equality follows from the symmetry $G(z,w)=G(w,z)$. 

It directly follows from inequality (\ref{3.2}) that
\begin{small}
\begin{align} \label{eqn: real-diagonal}
&\Re\int_{X\times X}G(z,w)(F(z,w)+\textbf{i}H(z,w))(F(w,z)+\textbf{i}H(w,z))dA(w)dA(z) \\ \notag
&\leq \int_{X\times X}G(z,w)|(F(z,w)+\textbf{i}H(z,w))|^2dA(w)dA(z).\notag
\end{align}
\end{small}

Recall that $Q$ is negative semi-definite. Part (1) of the conclusion follows from the inequality above: the right hand side estimate above for the third term in equation \eqref{eqn:Q(A,A)decomposition} cancels with the second term in equation \eqref{eqn:Q(A,A)decomposition}.

For part (2), beginning from the triangle inequality, we find
\begin{small}
\begin{align} \label{eqn:3.2final terms}
|Q(A,A)|&\leq 4\int_{X}{D (\Im\{F(z,z)+\textbf{i}H(z,z)\})\cdot(\Im\{F(z,z)+\textbf{i}H(z,z)\})dA(z)}\\ \notag
&+2\int_{X\times X}G(z,w)|F(z,w)+\textbf{i}H(z,w))|^2dA(w)dA(z)\\ \notag
&+2 |\Re\int_{X\times X}G(z,w)(F(z,w)+\textbf{i}H(z,w))(F(w,z)+\textbf{i}H(w,z))dA(w)dA(z)|\\ \notag
&\leq 4\int_{X}{D (\Im\{F(z,z)+\textbf{i}H(z,z)\})\cdot(\Im\{F(z,z)+\textbf{i}H(z,z)\})dA(z)}\\ \notag 
&+4\int_{X\times X}G(z,w)|F(z,w)+\textbf{i}H(z,w)|^2dA(w)dA(z) \quad (\text{by applying \eqref{eqn: real-diagonal}}) \\ \notag
&\leq 4\int_{X}{(\Im\{F(z,z)+\textbf{i}H(z,z)\})^2dA(z)}\\ \notag
&+4\int_{X\times X}G(z,w)|F(z,w)+\textbf{i}H(z,w)|^2dA(w)dA(z) \quad (\text{by Lemma~\ref{ufD}}) \\ \notag
&\leq  4\int_{X}{|F(z,z)+\textbf{i}H(z,z)|^2dA(z)}\\ \notag
&+4\int_{X\times X}G(z,w)|F(z,w)+\textbf{i}H(z,w)|^2dA(w)dA(z)\\ \notag
&\leq 8\cdot\int_{X}(|F(z,z)|^2+ |H(z,z)|^2)dA(z)\\ \notag
&+8\int_{X\times X}G(z,w)(|F(z,w)|^2+|H(z,w)|^2)dA(w)dA(z) \notag
\end{align}
\end{small}
with the last inequality following from the fact $G(z,w)\geq 0$ and the elementary inequality 
\[|z_{1}+\textbf{i}z_{2}|^2\leq 2(|z_{1}|^2+|z_2|^2), \ \ \forall z_{1},z_{2}\in \mathbb{C}.\]
\end{proof}

We use the first part of this proposition to prove Theorem~\ref{uubf-0} in the next section, and the second part of the proposition to prove Theorem~\ref{lbfme-0} in Section~5.

\section{A uniform upper bound for $\lambda_{\min}(X)$ }\label{s-ub}
From the definition of the curvature operator we know that 
\[\sum_{i=1}^{(3g-3)^2}\lambda_{i}(X)=\Sca(X).\]
Thus a trivial upper bound for $\lambda_{(3g-3)^2}(X)$ is $\lambda_{(3g-3)^2}(X) \le \frac{\Sca(X)}{(3g-3)^2}$. By Remark \ref{ubfs} in section 2, this upper bound is less than or equal to $\frac{-1}{12\pi(g-1)}$, which approaches zero as $g$ goes to infinity. The goal of this section is to prove Theorem~\ref{uubf-0} which estimates $\lambda_{(3g-3)^2}(X)$ by a negative number independent of genus.

 The proof is actually quite straightforward: because $\lambda_{\min}(X)$ will measure the $L^{\infty}$-norm of $\tilde{Q}$, we simply exhibit an (unit normed) element $A_0$ in $\wedge^2T_X\Teich(S_g)$ with $L^{\infty}$ norm bounded away from zero.  After a few preparatory remarks on how one estimates the norm of $\tilde{Q}$, we display our element $A_0$ and verify the claims as to its $\tilde{Q}$-norm.

There are three terms in the formula of the Weil-Petersson curvature operator presented in Proposition \ref{cof}. The first term is non-positive, and the sum of the second and third terms is also non-positive by the argument in the proof of Proposition \ref{upfco}.
Of course, since $\lambda_{(3g-3)^2}(X)$ is the minimum eigenvalue, the Rayleigh-Ritz formulation implies that for any element $A$, we have that  $\lambda_{(3g-3)^2}(X) \le \frac{Q(A,A)}{\|A\|_{eu}^2}$. Now we are ready to describe the proof of  Theorem \ref{uubf-0} in detail.

\begin{proof}[Proof of Theorem \ref{uubf-0}]
Let $\{\mu_{i}\}_{i=1}^{3g-3}$ be a holomorphic orthonormal basis $T_{X}\Teich(S_g)$ and let $\{\frac{\partial}{\partial t_{i}}\}_{i=1}^{3g-3}$ be the vector field on $\Teich(S_g)$ near $X$ such that $\frac{\partial}{\partial t_{i}}|_{X}=\mu_{i}$. Let $t_{i}=x_{i}+\textbf{i}y_{i}$ and let $A_0=\frac{1}{\sqrt{3g-3}}(\sum_{i=1}^{3g-3}\frac{\partial}{\partial x_{i}}(X)\wedge \frac{\partial}{\partial y_{i}}(X))$ be our special element of $\wedge^{2}T_{X}\Teich(S_g)$.

We begin with some preliminary computations on our special two-form $A_0$: we show that $\|A_0\|_{eu} =1$, and find the associated functions $F$ and $H$ used in the formula \eqref{eqn:Q(A,A)decomposition} for the associated form $Q(A_0,A_0)$ for this element. 
Since $\{\mu_{i}\}_{i=1}^{3g-3}$ is a holomorphic orthonormal basis $T_{X}\Teich(S_g)$, we have 
\begin{eqnarray*}
<\frac{\partial}{\partial x_{i}}(X),\frac{\partial}{\partial x_{j}}(X)>&=&\Re\{<\frac{\partial}{\partial x_{i}}(X)+\textbf{i}\textbf{J}\circ \frac{\partial}{\partial x_{i}}(X),\frac{\partial}{\partial x_{j}}(X)+\textbf{i}\textbf{J}\circ \frac{\partial}{\partial x_{j}}(X)>\}\\
&=&\Re\{<\mu_i, \mu_j>\}=\delta_{ij}.
\end{eqnarray*}
Similarly, we have
\begin{eqnarray*}
<\frac{\partial}{\partial y_{i}}(X),\frac{\partial}{\partial y_{j}}(X)>&=&\Re\{<\frac{\partial}{\partial y_{i}}(X)+\textbf{i}\textbf{J}\circ \frac{\partial}{\partial y_{i}}(X),\frac{\partial}{\partial y_{j}}(X)+\textbf{i}\textbf{J}\circ \frac{\partial}{\partial y_{j}}(X)>\}\\
&=&\Re\{<-\textbf{i}\mu_i, -\textbf{i}\mu_j>\}=\Re\{<\mu_i, \mu_j>\}=\delta_{ij}
\end{eqnarray*}

and

\begin{eqnarray*}
<\frac{\partial}{\partial x_{i}}(X),\frac{\partial}{\partial y_{j}}(X)>&=&\Re\{<\frac{\partial}{\partial x_{i}}(X)+\textbf{i}\textbf{J}\circ \frac{\partial}{\partial x_{i}}(X),\frac{\partial}{\partial y_{j}}(X)+\textbf{i}\textbf{J}\circ \frac{\partial}{\partial y_{j}}(X)>\}\\
&=&\Re\{<\mu_i, -\textbf{i}\mu_j>\}=-\Im\{<\mu_i, \mu_j>\}\\
&=&-\Im\{\delta_{ij}\}=0.
\end{eqnarray*}

Thus,  equation (\ref{eqn: mowedge}) gives that 
\begin{eqnarray*}
&&<\frac{\partial}{\partial x_{i}}(X)\wedge \frac{\partial}{\partial y_{i}}(X),\frac{\partial}{\partial x_{j}}(X)\wedge \frac{\partial}{\partial y_{j}}(X)>_{eu}\\
&=&<\frac{\partial}{\partial x_{i}}(X),\frac{\partial}{\partial x_{j}}(X)> <\frac{\partial}{\partial y_{i}}(X),\frac{\partial}{\partial y_{j}}(X)>\\
&-&<\frac{\partial}{\partial x_{i}}(X),\frac{\partial}{\partial y_{j}}(X)> <\frac{\partial}{\partial y_{i}}(X),\frac{\partial}{\partial x_{j}}(X)>\\
&=&\delta_{ij}^2-0=\delta_{ij}.
\end{eqnarray*}

Therefore, the norm of $A_0$ satisfies
\begin{equation}\label{eqn: noA_0}
||A_0||_{eu}^2=\frac{1}{3g-3}<\sum_{i=1}^{3g-3}\frac{\partial}{\partial x_{i}}(X)\wedge \frac{\partial}{\partial y_{i}}(X),\sum_{i=1}^{3g-3}\frac{\partial}{\partial x_{i}}(X)\wedge \frac{\partial}{\partial y_{i}}(X)>_{eu}=1.
\end{equation}

As we plan to apply Proposition~\ref{upfco}, we need to compute $F(z,z)$ and $H(z,z)$ for this element $A_0$.  Of course, 
\begin{equation*}
F(z,z)=\sum_{i,j=1}^{3g-3}{(a_{ij}+c_{ij})\mu_{i}(z)\cdot \overline{\mu_{j}(z)}} =0
\end{equation*}
as from our definition of $A_0$ and the coefficients $a_{ij}, c_{ij}$ in \eqref{eqn:defnA}, we see we have set all of the $a_{ij}= c_{ij}=0$ in the definition of $F(z,z)$ in Proposition~\ref{cof}. Similarly,
\begin{equation*}
H(z,w) = \sum_{i,j=1}^{3g-3}b_{ij}\mu_{i}(w)\cdot \overline{\mu_{j}(z)} = \sum_{i=1}^{3g-3}\frac{1}{\sqrt{3g-3}}\mu_{i}(w)\cdot \overline{\mu_{i}(z)}
\end{equation*}
by our definition of $H$ in Proposition~\ref{cof}.

From part (1) of Proposition \ref{upfco},
\begin{eqnarray*}
Q(A_0,A_0)&\leq& -4\int_{X}{D (\Im\{\textbf{i}H(z,z)\})\cdot(\Im\{\textbf{i}H(z,z)\})dA(z)}\\
&=& -4\int_{X}D(\frac{1}{\sqrt{3g-3}}\sum_{i=1}^{3g-3}|\mu_{i}|^2)(\frac{1}{\sqrt{3g-3}}\sum_{i=1}^{3g-3}|\mu_{i}|^2) dA\\
&\leq& -\frac{4}{3} \int_{X}(\frac{1}{\sqrt{3g-3}}\sum_{i=1}^{3g-3}|\mu_{i}|^2)^2 dA \quad (\text{by Lemma~\ref{lowolf}}) \\
&=&-\frac{4}{9(g-1)} \int_{X}(\sum_{i=1}^{3g-3}|\mu_{i}|^2)^2 dA\\
&\leq& \frac{2}{9(g-1)}\Sca(X),
\end{eqnarray*}
with the last inequality follows from Proposition \ref{bfscalar}. 

Combining equation (\ref{meve}), equation (\ref{eqn: noA_0}) and the inequality above, we have 
\[\lambda_{\min}(X)\leq \frac{Q(A_0,A_0)}{||A_0||_{eu}^2} \leq \frac{2 \Sca(X)}{9(g-1)}.\]

By Lemma 4.6 in \cite{Wol86}, Wolpert's upper bound for scalar curvature, we have 
\[\lambda_{\min}(X)\leq \frac{2 \Sca(X)}{9(g-1)}\leq \frac{2}{9(g-1)}\cdot \frac{-3(3g-2)}{4\pi} < \frac{-1}{2\pi}.\]
\end{proof}

\section{A uniform lower bound for $\lambda_{\min}(X)$ in the thick-part}\label{s-lb}
The goal of this section is to prove Theorem~\ref{lbfme-0} that was stated in the introduction.  In contrast to the argument in the last section where we bounded from above the norm of a particular element $A_0 \in \wedge^2T_X\Teich(S_g)$, here we need to bound from below the norm of an arbitrary element $A \in \wedge^2T_X\Teich(S_g)$.  The proof rests on the (other) bound in Proposition~\ref{upfco}, where here we inherit uniform bounds on all of the terms in the the norm $Q(A,A)$ from the uniform thickness of a surface in the thick part of the surface.

Recall the systole, denoted by $\systole(X)$, of a compact hyperbolic surface $X$ is the length of the shortest nontrivial simple closed curve in $X$. So $\systole(X)=2\inj(X)$ where $\inj(X)$ is the injectivity radius of $X$. The $\epsilon$-thick part $\Teich(S_g)^{\geq \epsilon}$ of $\Teich(S_g)$ is defined by 
\[\Teich(S_g)^{\geq \epsilon}:=\{Y\in \Teich(S_g); \systole(Y)\geq \epsilon\}.\] 
Let $\Mod(S_g)$ be the mapping class group of $S_g$. Then the quotient space $\mathbb{M}(S_g)^{\geq \epsilon}:=\Teich(S_g)^{\geq \epsilon}/\Mod(S_g)$ is called the $\epsilon$-thick part of the moduli space. By Mumford compactness, we have (see e.g. \cite{IT92-book}) that $\mathbb{M}(S_g)^{\geq \epsilon}$ is compact for all $\epsilon>0$. In particular there is a lower bound for $\lambda_{(3g-3)^2}(X)$ where $X$ runs over the thick part of $\Teich(S_g)$. A natural question is whether this lower bound depends on the topology of the surface. Theorem \ref{lbfme-0} answers this question negatively; specifically, we provide a lower bound that only depends on the thickness $\epsilon$ but not on genus of the surface.

We begin the proof of Theorem~\ref{lbfme-0} with some preliminary estimates on norms of sums of harmonic Beltrami differentials. Let $\{\mu_{i}\}_{i=1}^{3g-3}$ be a holomorphic orthonormal basis of $T_{X}\Teich(S_g)$; thus, in particular, we have that $\int_{X}\mu_{i}(z)\overline{\mu_{j}(z)}dA=\delta_{ij}$. Now, the expression of the Weil-Petersson curvature operator -- presented in Proposition~\ref{cof} --involved expressions $F(z,w)$ and $H(z,w)$ both of the same form, say $K(z,w)=\sum_{i,j=1}^{3g-3}{d_{ij}\mu_{i}(w)\cdot \overline{\mu_{j}(z)}}$ (where the coefficients $d_{ij}$ were real). We separate the proof of Theorem \ref{lbfme-0} into lemmas that estimate norms of such expressions $K(z,w)$. The following lemma is a direct consequence of Proposition \ref{lnubw}.

\begin{lemma}\label{glnbw}
Fix $z\in X$, then for $K(z,w)=\sum_{i,j=1}^{3g-3}{d_{ij}\mu_{i}(w)\cdot \overline{\mu_{j}(z)}}$, we have
\[\sup_{w\in X}|K(z,w)|^2\leq C(\inj(X))\cdot (\sum_{1\leq i,j,l\leq 3g-3}{d_{ij}d_{il}\overline{\mu_{j}(z)}\mu_{l}(z)})\]
where $C(\inj(X))$ is the same as in Proposition \ref{lnubw}.
\end{lemma}

\begin{proof}
Rewrite $K(z,w)$ as
\[K(z,w)=\sum_{i=1}^{3g-3}(\sum_{j=1}^{3g-3}{d_{ij}\overline{\mu_{j}(z)}})\cdot \mu_{i}(w).\] 

Thus, if we fix $z \in X$, the form $K(z,w)$ is a harmonic Beltrami differential on $X$ in the coordinate $w$. From Proposition \ref{lnubw}, we have
\begin{eqnarray*}
&&\sup_{w\in X}|K(z,w)|^2 \leq C(\inj(X))\cdot  \int_{X}{K(z,w)\overline{K(z,w)} dA(w)}\\
&=& C(\inj(X))\cdot  \int_{X}{(\sum_{i,j=1}^{3g-3}{d_{ij}\mu_{i}(w)\cdot \overline{\mu_{j}(z)}})(\sum_{k,l=1}^{3g-3}{d_{kl}\overline{\mu_{k}(w)}\cdot \mu_{l}(z)}) dA(w)}\\
&=& C(\inj(X))\cdot (\sum_{1\leq i,j,k,l\leq 3g-3} d_{ij}d_{kl} \overline{\mu_{j}(z)}\mu_{l}(z) )\cdot \int_{X}\mu_{i}(w)\overline{\mu_{k}(w)}dA(w))\\
&=& C(\inj(X))\cdot (\sum_{1\leq i,j,l\leq 3g-3}{d_{ij}d_{il}\overline{\mu_{j}(z)}\mu_{l}(z)}),
\end{eqnarray*} 
where the last equality uses the fact that the basis $\{\mu_{i}\}$ is orthonormal.
\end{proof}

Specializing to an $L^2$ bound for $K(z,z)$, we find

\begin{lemma}\label{iofg}
$\int_{X}|K(z,z)|^2dA(z)\leq C(\inj(X))\cdot (\sum_{1\leq i,j\leq (3g-3)}d_{ij}^2)$.
\end{lemma}
\begin{proof}
First, Lemma~\ref{glnbw} gives that, for any $z \in X$, 
\[|K(z,z)|^2\leq \sup_{w\in X}|K(z,w)|^2\leq C(\inj(X))\cdot (\sum_{1\leq i,j,l\leq 3g-3}{d_{ij}d_{il}\overline{\mu_{j}(z)}\mu_{l}(z)}).\]
Thus, we have
\begin{eqnarray*}
&&\int_{X}|K(z,z)|^2dA(z)\\
&\leq&C(\inj(X))\cdot \int_{X}{\sum_{1\leq i,j,l\leq 3g-3}{d_{ij}d_{il}\overline{\mu_{j}(z)}\mu_{l}(z)}dA(z)}\\
&=& C(\inj(X))\cdot (\sum_{1\leq i,j\leq (3g-3)}d_{ij}^2)
\end{eqnarray*} 
where the last equality follows from the assumption that $\{\mu_{i}\}$ is orthonormal.
\end{proof}

\begin{lemma}\label{iofg-2}
$\int_{X\times X}G(z,w) |K(z,w)|^2dA(w)dA(z)\leq C(\inj(X))\cdot (\sum_{1\leq i,j\leq (3g-3)}d_{ij}^2)$.
\end{lemma}

\begin{proof}
Since the Green function $G(z,w)\geq 0$ for all $(z,w)\in X\times X$, from Lemma \ref{glnbw} we have
\begin{eqnarray*}
&&\int_{X\times X}G(z,w) |K(z,w)|^2dA(w)dA(z) \\
&\leq& \int_{X\times X}G(z,w) (\sup_{w \in X}|K(z,w)|^2 )dA(w)dA(z) \\
&\leq&C(\inj(X))\cdot\int_{X\times X}G(z,w){\sum_{1\leq i,j,l\leq 3g-3}{d_{ij}d_{il}\overline{\mu_{j}(z)}\mu_{l}(z)}dA(w)dA(z)}\\
&=& C(\inj(X))\cdot\int_{X} D(1) \cdot {(\sum_{1\leq i,j,l\leq 3g-3}{d_{ij}d_{il}\overline{\mu_{j}(z)}\mu_{l}(z))} dA(z)}   \\
&=& C(\inj(X))\cdot (\sum_{1\leq i,j\leq (3g-3)}d_{ij}^2)
\end{eqnarray*} 
where the last equality follows from that face $D(1)=1$ and the assumption that $\{\mu_{i}\}$ is orthonormal.
\end{proof}

Now we are ready to prove Theorem \ref{lbfme-0}.
\begin{proof}[Proof of Theorem \ref{lbfme-0}]
 Let $A \in \wedge^{2}T_{X}\Teich(S_g)$ be expressed as 
\[A=(\sum_{1\leq i<j\leq (3g-3)}a_{ij}\frac{\partial}{\partial x_{i}}\wedge \frac{\partial}{\partial x_{j}}+\sum_{1\leq i,j\leq (3g-3)}b_{ij} \frac{\partial}{\partial x_{i}}\wedge \frac{\partial}{\partial y_{j}}+\sum_{1\leq i<j\leq (3g-3)}c_{ij}\frac{\partial}{\partial y_{i}}\wedge \frac{\partial}{\partial y_{j}})\] 
where $a_{ij}, b_{ij},c_{ij}$ are real. In terms of the coefficients of the basic forms, we may compute the norm $\|A\|_{eu}$ of $A$ as
\begin{equation} \label{eqn:A norm}
\|A\|_{eu}^2 = \sum_{1\leq i<j\leq (3g-3)}a_{ij}^2+\sum_{1\leq i,j\leq (3g-3)}b_{ij}^2+\sum_{1\leq i<j\leq (3g-3)}c_{ij}^2.
\end{equation}
From part (2) of Proposition~\ref{upfco}, we have
\begin{eqnarray}\label{11-0}
-Q(A,A) &\leq& 8\cdot\int_{X}(|F(z,z)|^2 + |H(z,z)|^2)dA(z)+\\
&8&\cdot \int_{X\times X}G(z,w)(|F(z,w)|^2+|H(z,w)|^2)dA(w)dA(z) \notag
\end{eqnarray} 
where $F(z,z)=\sum_{1\leq i<j\leq (3g-3)}{(a_{ij}+c_{ij})\mu_{i}(z)\cdot \overline{\mu_{j}(z)}}$ and $H(z,z)=\sum_{i,j=1}^{3g-3}{b_{ij}\mu_{i}(z)\cdot \overline{\mu_{j}(z)}}$. 

By Lemma \ref{iofg} and Lemma \ref{iofg-2}, we have
\begin{eqnarray}\label{11-1}
\int_{X}|F(z,z)|^2dA(z)\leq  C(\inj(X))\cdot (\sum_{1\leq i<j\leq (3g-3)}(a_{ij}+c_{ij})^2)
\end{eqnarray}
and
\begin{eqnarray}\label{11-1-1}
\int_{X\times X}G(z,w)|F(z,w)|^2dA(w)dA(z)\\
\leq  C(\inj(X))\cdot (\sum_{1\leq i<j\leq (3g-3)}(a_{ij}+c_{ij})^2)\nonumber
\end{eqnarray}
where we have set $\{d_{ij}=a_{ij}+c_{ij}\}$ for $i<j$ and $d_{ij}=0$ otherwise.

Similarly, by Lemma \ref{iofg} and Lemma \ref{iofg-2}, we also have
\begin{eqnarray}\label{11-2}
\int_{X}|H(z,z)|^2dA(z)\leq  C(\inj(X))\cdot (\sum_{1\leq i,j\leq (3g-3)}b_{ij}^2)
\end{eqnarray}
and
\begin{eqnarray}\label{11-2-1}
\int_{X\times X}G(z,w)|H(z,w)|^2dA(w)dA(z)\\
\leq   C(\inj(X))\cdot (\sum_{1\leq i,j\leq (3g-3)}b_{ij}^2) \nonumber
\end{eqnarray}
by setting $\{d_{ij}=b_{ij}\}$.

From inequalities (\ref{11-0})-(\ref{11-2-1}), we find
\begin{eqnarray*}
Q(A,A) &\geq& -16 C(\inj(X))\cdot (\sum_{1\leq i<j\leq (3g-3)}(a_{ij}+c_{ij})^2+\sum_{1\leq i,j\leq (3g-3)}b_{ij}^2)\\
&\geq& -32 C (\inj(X))\cdot (\sum_{1\leq i<j\leq (3g-3)}(a_{ij}^2+c_{ij}^2)+\sum_{1\leq i,j\leq (3g-3)}b_{ij}^2)\\
&=& -32C (\inj(X))\|A\|_{eu}^2 \text{ by \eqref{eqn:A norm} }.
\end{eqnarray*} 
The theorem then follows from the Rayleigh-Ritz characterization of the lowest eigenvalue and by choosing $B(\epsilon)=32C(\frac{\epsilon}{2})$.
\end{proof}

\begin{proof}[Proof of Theorem \ref{ubfl^p}]
For $1\leq i \leq (3g-3)^2$, let $\lambda_{i}(X_g)$ be the $i$-th eigenvalue of the Weil-Petersson curvature operator. From Theorem \ref{conp} we know that $\lambda_i = \lambda_{i}(X_g)<0$. So we have 
$$||\tilde{Q}||^p_{\ell^p}(X_g)=\sum_{i=1}^{(3g-3)^2}(-\lambda_{i})^p \leq \sum_{i=1}^{(3g-3)^2}(-\lambda_{i})\cdot (-\lambda_{(3g-3)^2})^{p-1}.$$

Since $p>1$, Theorem \ref{lbfme-0} gives that
$$||\tilde{Q}||^p_{\ell^p}(X_g) \leq \sum_{i=1}^{(3g-3)^2}(-\lambda_{i})\cdot B(\epsilon)^{p-1}= B(\epsilon)^{p-1} \cdot (-\Sca(X_g)).$$

Recall that Proposition 3.3 in Teo's paper \cite{Teo08} says that $\Sca(X_g)\geq -6(g-1)C(\frac{\epsilon}{2})$. So we have

$$||\tilde{Q}||^p_{\ell^p}(X_g) \leq B(\epsilon)^{p-1} \cdot (6(g-1)C(\frac{\epsilon}{2})).$$

In the proof of Theorem \ref{lbfme-0} we choose $B(\epsilon)=32C(\frac{\epsilon}{2})$. Thus,

$$||\tilde{Q}||^p_{\ell^p}(X_g) \leq 32^{p-1}\cdot (C(\frac{\epsilon}{2}))^p \cdot 6(g-1)\leq (B(\epsilon))^p \cdot g.$$

The conclusion follows by taking a $p$-th root.

\end{proof}

\begin{remark}
Theorem \ref{lbfme-0} shows the Weil-Petersson curvature operators, restricted to the thick part of the moduli space, are uniformly bounded by a constant that depends only on the bound on prescribed thickness (the lower bound in injectivity radius) of the surfaces in the given thick part, but not on the topology of the underlying surface.  It is worth noting that the constant $B(\epsilon) \asymp \frac{1}{\epsilon^2}$ tends to infinity as $\epsilon \to 0$.  Of course, it is known that sectional curvatures (the diagonal elements for the curvature operator $\tilde{Q}$) decay at worst with order $O(\frac{-1}{\epsilon})$. (One can see \cite{Huang1} and \cite{Wolpert5} for details.) We do not know if this bound of $O(\frac{-1}{\epsilon})$, sharp in the case of sectional curvatures, also extends to the broader setting of the curvature operator. On the other hand, as $\epsilon \to \infty$, the constant $B(\epsilon) \to \frac{24}{\pi} > \frac{1}{2\pi}$, which is consistent with the bound we found for a special element of $\wedge^2T_{X}\Teich(S_g)$ in Theorem~\ref{uubf-0}.
\end{remark}

Let $R$ be the Riemannian curvature tensor of the Weil-Petersson metric. Let $X \in \Teich(S_g)$ and $\{e_{i}\}_{i=1}^{6g-6}$ be an orthonormal basis of  $T_X\Teich(S_g)$. Recall that $R_{ijkl}:=<R(e_{i},e_{j})e_{k},e_{l}>$. Define the norm $||R||_{X}$ of the curvature tensor $R$ on $X$ by 
\[||R||_{X}=\sup\max_{1\leq i,j,k,l\leq (6g-6)}{|R_{ijkl}|}\]   
where the supremum runs over all the orthonormal bases of $T_X\Teich(S_g)$.

A simple application of Theorem \ref{lbfme-0} is 
\begin{corollary}\label{avfob}
Let $X\in \Teich(S_g)$ and $R$ be the Riemannian curvature tensor of the Weil-Petersson metric. Then,
\[||R||_{X}\leq \widehat{B}(\inj(X))\]
where $\inj(X)$ is the injectivity radius of $X$ and $\widehat{B}(\inj(X))$ is a function of $\inj(X)$.
\end{corollary}

\begin{proof}
Set $||Q(X)||:=\sup_{A\in \wedge^2 T_X\Teich(S_g),\|A\|_{eu}=1}|<\tilde{Q}(A),A>|$. Theorem~\ref{lbfme-0} gives that $||Q(X)|| \leq B(2\inj(X))$.
Let $\{e_{i}\}_{i=1}^{6g-6}$ be an orthonormal basis of  $T_X\Teich(S_g)$. Then the Cauchy Schwarz inequality leads to
\begin{eqnarray*}
|R_{ijkl}|&=&|<\tilde{Q}(e_{i}\wedge e_{j}),e_{k}\wedge e_{l}>|\\
&\leq& |\tilde{Q}(e_{i}\wedge e_{j})|\cdot |e_{k}\wedge e_{l}|\leq||Q(X)|| \leq B(2\inj(X)).
\end{eqnarray*} 
Since the indices $\{i,j,k,l\}$ are arbitrary, the conclusion follows by choosing $\widehat{B}(\inj(X))= B(2\inj(X))$.
\end{proof}

\begin{remark}
For the case that $\{i=k,j=l\}$, $R_{ijij}$ is the sectional curvature (using the orthonormality of the coordinate system). Huang in \cite{Huang} proved the the sectional curvature, restricted on the thick part of the moduli space, is uniformly bounded from below by using harmonic maps.
\end{remark}

The Weil-Petersson metric is not complete (see \cite{Chu, Wol75}). We let $\overline{\Teich(S_g)}$ be the metric completion of $\Teich(S_g)$ and $\partial\overline{\Teich(S_g)}$ be the frontier of $\overline{\Teich(S_g)}$: this frontier is composed of products of lower dimensional Teichm\"uller spaces (one can see \cite{Mas76, Wolpert3, Wolpert4} for more details). For any $X\in \Teich(S_g)$, the following quantitative version of estimation on the distance between $X$ and $\partial\overline{\Teich(S_g)}$ is provided by Wolpert. 

\begin{theorem}[\cite{Wol08}]\label{eowfd}
$\dist(X,\partial\overline{\Teich(S_g)})\leq \sqrt{4\pi \cdot \inj(X)}.$
\end{theorem}
The following consequence is motivated by Proposition 4.22 in \cite{BHW12}.
\begin{corollary}\label{cbbd}
There exists a constant $C>0$, which is independent of the topology of the surface, such that the norm $||R||_{X}$ of the curvature tensor $R$ on $X$ satisfies
\[||R||_{X}\leq \max\{C,\frac{C}{\dist(X,\partial\overline{\Teich(S_g)})^4}\}.\]
\end{corollary}

\begin{proof}

We observed after stating Proposition~\ref{lnubw} that as $\epsilon \to 0$, the constant $C(\epsilon)$ is asymptotic to $\frac{1}{\pi \epsilon^2}$.
Fix $\epsilon_{0}>0$ such that $C(\epsilon)\leq \frac{2}{\pi \epsilon^2}$ for all $\epsilon \in (0,\epsilon_0]$. Then for any $X\in \Teich(S_g)$, there are two possibilities: 

Case 1: If the injectivity radius $\inj(X)\geq \epsilon_0$, the inequality $||R||_{X}\leq \widehat{B}(\epsilon_0)=B(2\epsilon_0)$ follows from Corollary \ref{avfob}.

Case 2:  If the injectivity radius $\inj(X)\leq \epsilon_0$, then from Corollary~\ref{avfob} we find,
\begin{eqnarray*}
||R||_{X}\leq 32 C(\inj(X))\leq \frac{64}{\pi \cdot \inj(X)^2} \leq \frac{1024\pi}{\dist(X,\partial\overline{\Teich(S_g)})^4},
\end{eqnarray*}
where we apply Theorem \ref{eowfd} for the last inequality. The conclusion follows by choosing $C=\max\{B(2\epsilon_0), 1024\pi\}$.
\end{proof}

\begin{remark}
We emphasize here that the constant $C$ is independent not only of the genus $g$, but also of any choice of neighborhood of $\partial\overline{\Teich(S_g)}$.
\end{remark}

\section{Eigenvalues of the Weil-Petersson curvature operators on thick surfaces}\label{zpoev}

Our goal in this final section of this part of the paper is to prove the remaining results, Theorem~\ref{largei-0} and Theorem~\ref{fftz-0}, on the \wep curvature operator. The proofs are reasonably straightforward consequences of the results and tools we have already developed.

We first prove Theorem~\ref{largei-0} which bounds the smallest eigenvalue $\lambda_{\min}$ of the Weil-Petersson curvature operator on surfaces of sufficiently large genus whose injectivity radii are increasing without bound.

\begin{proof}[Proof of Theorem \ref{largei-0}]
First, one computes that the constant $B(\epsilon)$ in Theorem~\ref{lbfme-0} satisfies
\begin{eqnarray*}\label{llimit}
\lim_{\epsilon \to \infty} B(\epsilon)=\frac{24}{\pi}.
\end{eqnarray*}
The conclusion then follows from Theorem \ref{uubf-0}, Theorem \ref{lbfme-0} and the limit above.
\end{proof}

Next, fix $\epsilon>0$; then the Mumford compactness 
theorem implies the $\epsilon$-thick part $\M^{\ge \epsilon}$ of the moduli space of a closed surface is compact. Thus the following functions are well-defined:
\[\underline{\lambda}^{\epsilon}_{i}(g):=\min_{X\in Teich(S_g)^{\geq \epsilon}}{\lambda_{i}(X)}, \ \  \forall 1\leq i \leq (3g-3)^2\]
and
\[\overline{\lambda}^{\epsilon}_{i}(g):=\max_{X\in Teich(S_g)^{\geq \epsilon}}{\lambda_{i}(X)}, \ \  \forall 1\leq i \leq (3g-3)^2.\]
Of course
\[\underline{\lambda}^{\epsilon}_{i}(g) \leq \overline{\lambda}^{\epsilon}_{i}(g),\]
and Theorems~\ref{uubf-0} and \ref{lbfme-0} imply
\[-B(\epsilon)\leq \underline{\lambda}^{\epsilon}_{(3g-3)^2}(g) \leq \overline{\lambda}^{\epsilon}_{(3g-3)^2}(g)\leq \frac{-1}{2\pi}.\]

Our focus for the rest of this section then turns from the smallest eigenvalue $\lambda_{(3g-3)^2}= \lambda_{\min}$ to the asymptotic properties for $\underline{\lambda}^{\epsilon}_{i}(g)$ and $\overline{\lambda}^{\epsilon}_{i}(g)$ as $g$ goes to infinity where the index $i$ is not maximal but instead small.

 Let
 \[i:\{1,2,\cdots,(3g-3)^2\} \rightarrow \{1,2,\cdots,(3g-3)^2\}\]
be a function. Recall that Theorem~\ref{fftz-0} states that if the index $i$ is not close to the maximal value $(3g-3)^2$ in the limit sense, then the bound $\underline{\lambda}^{\epsilon}_{i}(g)$ tends to zero as $g$ goes to infinity. 

\begin{proof}[Proof of Theorem \ref{fftz-0}]
Since $1\leq i(g)\leq (3g-3)^2$, It follows from Theorem \ref{(3g-3)^2} that $\overline{\lambda}^{\epsilon}_{i(g)}(g)<0$. 
For the lower bound, let $X \in \Teich(S_g)^{\geq \epsilon}$ and recall that $\{\lambda_{i}(X)\}_{i=1}^{(3g-3)^2}$ is the set of the non-zero eigenvalues of the Weil-Petersson curvature operator at $X$ with $\lambda_{i+1}(X)\leq\lambda_{i}(X)$. Proposition~3.3 in \cite{Teo08} tells us that 
\begin{eqnarray*}
-6(g-1)C(\frac{\epsilon}{2})\leq \Sca(X) =\sum_{i=1}^{(3g-3)^2}{\lambda_{i}(X)},
\end{eqnarray*} 
where $C(\frac{\epsilon}{2})$ is defined in Proposition~\ref{lnubw}. Since $\lambda_{i}(X)<0$, we have that
\begin{eqnarray*}
-6(g-1)C(\frac{\epsilon}{2}) \leq \sum_{i=f(g)}^{(3g-3)^2}{\lambda_{i}(X)}\leq ((3g-3)^2-i(g)+1)\lambda_{f(g)}(X).
\end{eqnarray*} 
We rewrite this last inequality as
\begin{eqnarray*}
\frac{-6C(\frac{\epsilon}{2})}{g}\cdot \frac{g(g-1)}{(3g-3)^2-i(g)+1}\leq \lambda_{f(g)}(X).
\end{eqnarray*} 
Then, since $\limsup_{g\to \infty} \frac{g(g-1)}{(3g-3)^2-i(g)+1}=\frac{1}{9(1-\alpha)}>0$ and $X\in \Teich(S_g)^{\geq \epsilon}$ is arbitrary, the conclusion follows by choosing $B(\alpha,\epsilon)=\frac{4C(\frac{\epsilon}{2})}{3(1-\alpha)}$.
\end{proof}

\begin{remark}

(1). Recall that as $\epsilon \to 0$, the constant $C(\epsilon)$ behaves like 
\[\lim_{\epsilon \to 0}\frac{C(\epsilon)}{\frac{1}{\pi \epsilon^2}}=1.\]
The proof of Theorem \ref{fftz-0} then yields the following more general statement.
\begin{theorem}\label{fffz-0-1}

If the function $f$ satisfies $\limsup_{g\to \infty}\frac{f(g)}{9g^2}=\alpha<1$. Let $X_g$ be a sequence of Riemannian surfaces whose injectivity radii satisfy $\lim_{g\to \infty}\frac{1}{g\cdot \inj(X_g)^2}=0$. Then,
\[\lim_{g\to \infty} \lambda_{f(g)}(X_g)=0.\] 

\end{theorem}

This theorem tells us that as the genus goes to infinity, even in certain portions of the thin parts of the moduli spaces, we could still have 
\[\lim_{g\to \infty} \overline{\lambda}^{\epsilon(g)}_{i}(g)=\lim_{g\to \infty} \underline{\lambda}^{\epsilon(g)}_{i}(g)=0, \ \  for \ all \ \ 1\leq i \leq 8g^2\]
provided that $\lim_{g\to \infty}\frac{1}{g\cdot \epsilon(g)^2}=0.$ \\

(2). \textsl{We do not know any asymptotic properties of $\underline{\lambda}^{\epsilon}_{i}(g)$ and $\overline{\lambda}^{\epsilon}_{i}(g)$ as $g$ goes to infinity when $i$ is close to $(3g-3)^2$; for example $\underline{\lambda}^{\epsilon}_{(3g-3)^2-g}(g)$ and $\overline{\lambda}^{\epsilon}_{(3g-3)^2-g}(g)$.  It is interesting to study them.}
\end{remark}

\part{Uniform bounds on the Weil-Petersson holomorphic sectional curvatures}

\section{Uniform estimates on harmonic Beltrami differentials.}\label{hmsc}
Recall from section~\ref{sec:suff thick} that a surface is $C$-sufficiently thick if $\inj(X) \ge C$, and that there are non-trivial subspaces of the Teichm\"uller spaces $\Teich(S_g)$ composed of sufficiently thick surfaces, once the genus $g$ is chosen sufficiently large. In this section will construct holomorphic lines
in the tangent spaces of sufficiently thick surfaces (of course, necessarily of large genera and injectivity radii) such that those sections have uniformly pinched negative holomorphic sectional curvatures. The pinching constants will be independent of the genus.

To begin, we pick up the thread on which we ended section~\ref{subsec:hol-sec-curv}, that related holomorphic sectional curvatures to powers of normalized $L^{\infty}$ norms of harmonic Beltrami differentials.  In particular, we will directly apply Proposition~\ref{ratio} to a specific choice of harmonic Beltrami differential $\mu_g$, whose construction we begin immediately below.  As we noted at the end of section~\ref{subsec:hol-sec-curv}, our (main) goal is to bound the $L^2$ norm of $\mu_g$ from above and the $L^\infty$ norm of $\mu_g$ from below.

\subsection{Poincar\'e series} Let $X_g \in \Teich(S_g)$ and let $\Gamma_g \subset \Aut(\D)$ be a representation of $\pi_1(X_g)$ into the (Mobius) group $\Aut(\D)$ of automorphisms of the hyperbolic disk $\D$ where $\pi: \D\rightarrow X_g$ is the universal covering map. Let $H(\D)$ denote the set of holomorphic functions on the unit disk $\D$ and $H^0(X_g, K^2)$ be the space of holomorphic quadratic differentials on $X_g$.

Recall the Theta-operator $\Theta$ is defined as
\begin{eqnarray*}
\Theta: \ H(\D)&\rightarrow& H^0(X_g, K^2)\\
              f(z) &\mapsto& \sum_{\gamma \in \Gamma_g} f(\gamma(z))\cdot \gamma'(z)^2.
\end{eqnarray*}

It is well-known that $\Theta (f)(z)$ is well-defined if $f$ is integrable on $\D$, and this operator is surjective on its domain. (See Theorem 7.2 in \cite{IT92-book} for details.)

Fix the constant function $f(z)\equiv 1$ be on $\D$. Then
\[\Theta(f(z)) = \Theta(1)(z)=\sum_{\gamma \in \Gamma_g} \gamma'(z)^2\]
whose corresponding harmonic Beltrami differential is

\begin{equation}\label{tps}
\mu_g(z):=\frac{\overline{\Theta(1)(z)}}{\rho(z)}=\sum_{\gamma \in \Gamma_g} \frac{\overline{\gamma'(z)^2}}{\rho(z)}
\end{equation}
where $\rho(z)=\frac{4}{(1-|z|^2)^2}$ is the hyperbolic metric on the disk.

Ahlfors in \cite{Ahlfors64} showed that $|\mu_g|$ is bounded above by a constant depending on $\Gamma_g$. In this section we will observe that $|\mu_g|$ is uniformly bounded above if we assume that the injectivity radius $\inj(X_g) \to \infty$ as $g\to \infty$.

Let us close this subsection by recalling Ahlfors' method in \cite{Ahlfors64}, a technique which we will adapt for the heart of our argument.

\textbf{Ahlfors' Method:} From the triangle inequality we know that 
\begin{equation} \label{eqn:ub-mu}
|\mu_g(z)|\leq \sum_{\gamma \in \Gamma_g} \frac{|\gamma'(z)|^2}{\rho(z)}.
\end{equation}
Then since $\rho(\gamma(z))|\gamma'(z)^2|=\rho(z)$ for any $\gamma \in \Gamma_g$, and $\rho(\zeta) = 4(1-|\zeta|^2)^{-2}$, we have

\begin{equation} \label{eqn:gamma-sum}
\sum_{\gamma \in \Gamma_g} \frac{|\gamma'(z)|^2}{\rho(z)}=\frac{1}{4} \sum_{\gamma \in \Gamma_g} (1-|\gamma(z)|^2)^2.
\end{equation}
Combining the above inequalities yields
\begin{equation} \label{eqn:ub-mu-1}
|\mu_g(z)|\leq\frac{1}{4} \sum_{\gamma \in \Gamma_g} (1-|\gamma(z)|^2)^2.
\end{equation}

Let $\Delta$ be the (Euclidean) Laplace operator on the (Euclidean) disk. Then a direct computation shows

\[\Delta(\sum_{\gamma \in \Gamma_g} (1-|\gamma(z)|^2)^2)=8\cdot \sum_{\gamma \in \Gamma_g} (2|\gamma(z)|^2-1)|\gamma'(z)|^2.\]

Note that the terms on the right side are non-negative when $|\gamma(z)|^2 \ge \frac{1}{2}$.  With that in mind, let $B(0,\frac{1}{\sqrt{2}}):=\{z\in \D; \  |z|\leq \frac{1}{\sqrt{2}}\}$ 
be the ball of Euclidean radius $\frac{1}{\sqrt{2}}$
and let
$V :=\cup_{\gamma \in \Gamma_g}\gamma^{-1}\circ B(0,\frac{1}{\sqrt{2}})$ be the pullbacks of this ball $B$ by the group $\Gamma_g$. The equation above gives that $\sum_{\gamma \in \Gamma_g} (1-|\gamma(z)|^2)^2$ is subharmonic in $\D-V$. Since both $\sum_{\gamma \in \Gamma_g} (1-|\gamma(z)|^2)^2$ and $V$ are $\Gamma_g$-invariant, and $\Gamma_g$ is cocompact, we find

\begin{equation}\label{cis}
\sup_{z\in \D}\sum_{\gamma \in \Gamma_g} (1-|\gamma(z)|^2)^2=\sup_{z\in V}\sum_{\gamma \in \Gamma_g} (1-|\gamma(z)|^2)^2=\sup_{z\in B(0,\frac{1}{\sqrt{2}})}\sum_{\gamma \in \Gamma_g} (1-|\gamma(z)|^2)^2
\end{equation}
which in particular is bounded above by a constant depending on $\Gamma_g$: we display a version of the argument that the last term on the right is bounded in the next subsection: see especially equations \eqref{cis-2} - \eqref{cis-4-1}.

\subsection{Uniform upper bound for $\mu_g$} Recall the relation between the Euclidean distance and the hyperbolic distance is 
\[dist_{\D}(0,z)=\ln{\frac{1+|z|}{1-|z|}}.\]

For the rest of this section, we consider a surface $X_g$ which is sufficiently thick -- we retain the index $g$ as a reminder that such surfaces have large genus $g$ and we are computing curvatures of the \tec space $\Teich(S_g)$ with all of the normalizations that accompany that choice. (This will also be useful for a remark at the end of this section.) Therefore we may assume that we may find a ball $B(0,\frac{1}{\sqrt{2}}+\frac{1}{10}) \subset F_g$  where $F_g\subset \D$ is the (Dirichlet) fundamental domain, centered at the origin, of $X_g$. Applying the triangle inequality to equation (\ref{cis}), we have

\begin{equation}\label{cis-1}
\sup_{z\in \D}\sum_{\gamma \in \Gamma_g} (1-|\gamma(z)|^2)^2\leq 1+\sup_{z\in B(0,\frac{1}{\sqrt{2}})}\sum_{\gamma \neq e} (1-|\gamma(z)|^2)^2.
\end{equation}

As in Ahlfors' method, by applying the Laplace operator, we obtain

\begin{equation}\label{cis-2}
\Delta(\sum_{\gamma \neq e} (1-|\gamma(z)|^2)^2)=8\cdot \sum_{\gamma \neq e} (2|\gamma(z)|^2-1)|\gamma'(z)|^2.
\end{equation}

Because of our assumption on the large injectivity radius of $X_g$, we see that $B(0,\frac{1}{\sqrt{2}}+\frac{1}{10}) \subset F_g$. Thus for 
$z\in B(0,\frac{1}{\sqrt{2}})$, we see that 
$\gamma(z)\notin B(0,\frac{1}{\sqrt{2}})$ when $\gamma \ne e$: in particular, for such $\gamma$, we have $|\gamma(z)|^2 \ge \frac{1}{2}$. This implies that for
$z\in B(0,\frac{1}{\sqrt{2}}+ \frac{1}{10})$
we have

\begin{equation}\label{cis-3}
\Delta(\sum_{\gamma \neq e} (1-|\gamma(z)|^2)^2)\geq 0.
\end{equation}

Let $\zeta=\xi+\bf{i}\eta $. Thus, by the mean-value inequality, for any $z\in B(0,\frac{1}{\sqrt{2}})$, we have

\begin{eqnarray*}
\sum_{\gamma \neq e} (1-|\gamma(z)|^2)^2 &\leq& \frac{1}{B(z,\frac{1}{10})}\int_{B(z,\frac{1}{10})}\sum_{\gamma \neq e} (1-|\gamma(\zeta)|^2)^2d\xi d\eta\\
&=& \frac{100}{\pi}\int_{B(z,\frac{1}{10})}\sum_{\gamma \neq e} \frac{|\gamma'(\zeta)|^2}{\rho(\zeta)}d\xi d\eta \nonumber    \text{ by equation } \eqref{eqn:gamma-sum} \\
&\leq & \frac{25}{\pi}\int_{B(z,\frac{1}{10})}\sum_{\gamma \neq e} |\gamma'(\zeta)|^2d\xi d\eta \text{ since } \rho(z) = 4(1-|z|^2)^{-2} \ge 4 \nonumber\\
&=& \frac{25}{\pi}\int_{\cup_{\gamma \neq e} \gamma\circ B(z,\frac{1}{10})}d\xi d\eta \text{ after unfolding the sum }\nonumber\\
&\leq & \frac{25}{\pi}\int_{\D-F_g}d\xi d\eta \nonumber\\
&=& \frac{25}{\pi}(\pi-Area(F_g)). \nonumber
\end{eqnarray*}

Since $z\in B(0,\frac{1}{\sqrt{2}})$ is arbitrary,

\begin{eqnarray}\label{cis-4-1}
\sup_{z\in B(0,\frac{1}{\sqrt{2}})}\sum_{\gamma \neq e, \ \gamma \in \Gamma_g} (1-|\gamma(z)|^2)^2  \leq \frac{25}{\pi}(\pi-\Area(F_g)).
\end{eqnarray}

Finally, we invoke the hypothesis that the surface $X_g$ is sufficiently thick. Assuming that thickness, we may assume that the Euclidean area $\Area(F_g)$ of the fundamental domain $F_g$ is as close to $\pi$ as we wish, so the quantitative version of sufficiently thick that we will invoke will force
\begin{equation} \label{cis-4-2}
\frac{25}{\pi}(\pi-\Area(F_g)) < \frac{1}{4}.
\end{equation}
Then, this last equation, together with \eqref{cis-4-1}, provides

\begin{equation}\label{cis-5}
\sup_{z\in B(0,\frac{1}{\sqrt{2}})}\sum_{\gamma \neq e, \ \gamma \in \Gamma_g} (1-|\gamma(z)|^2)^2 < \frac{1}{4}.
\end{equation}

\begin{remark}
For the estimate \eqref{cis-4-2} to hold, we need the ball, centered at the origin, to have Euclidean radius $\sqrt{\frac{99}{100}}$, so that the surface is at least $\ln{\frac{10+\sqrt{99}}{10-\sqrt{99}}}$-thick.
\end{remark}

The following two propositions follow easily from the equations above.

\begin{proposition}\label{uubfps}
Let $X_g$ be sufficiently thick.  Then
\[\sup_{z\in \D}|\mu_g(z)|< \frac{5}{16}.\]
\end{proposition}

\begin{proof}
We compute
\begin{align*}
|\mu_g(z)| &\leq \sum_{\gamma \in \Gamma_g} \frac{|\gamma'(z)^2|}{\rho(z)} \text{ by } \eqref{tps}\\
&= \frac{1}{4} \sum_{\gamma \in \Gamma_g} (1-|\gamma(z)|^2)^2 \text{ by } \eqref{eqn:gamma-sum} \\
&\leq \frac{1}{4} \sup_{z\in \D}\sum_{\gamma \in \Gamma_g} (1-|\gamma(z)|^2)^2\\
&= \frac{1}{4} \sup_{z\in B(0,\frac{1}{\sqrt{2}})}\sum_{\gamma \in \Gamma_g} (1-|\gamma(z)|^2)^2   \text{ by } \eqref{cis} \\
&\leq \frac{1}{4} +\frac{1}{4} \sup_{z\in B(0,\frac{1}{\sqrt{2}})} \sum_{\gamma \ne e, \gamma \in \Gamma_g} (1-|\gamma(z)|^2)^2 \\
&< \frac{5}{16} \text{ by } \eqref{cis-5}. 
\end{align*}
The conclusion follows by taking a supremum.
\end{proof}

\begin{proposition}\label{ulbfps}
Let $X_g$ be sufficiently thick as above. Then
\[|\mu_g(0)|\geq \frac{3}{16}.\]
\end{proposition}

\begin{proof}
By the triangle inequality we have 
\begin{align*}|\mu_g(0)| &\geq \frac{1}{4}-\sum_{\gamma \neq e \in \Gamma_g} \frac{|\gamma'(0)|^2}{\rho(0)} \\
& =\frac{1}{4}-\frac{1}{4}\sum_{\gamma \neq e \in \Gamma_g} (1-|\gamma(0)|^2)^2 \\
& > \frac{3}{16} \text{ by } \eqref{cis-5}.
\end{align*}
\end{proof} 

\begin{remark}
We have chosen our setting to be one of a single sufficiently thick surface.  However, we could also consider a sequence $\{X_g\}$ of surfaces, say indexed by genus, whose injectivity radii $\inj(X_g)$ is growing (without bound) with the genus $g$. In that case, the estimate \eqref{cis-4-1} becomes
$$
\sup_{z\in B(0,\frac{1}{\sqrt{2}})}\sum_{\gamma \neq e, \ \gamma \in \Gamma_g} (1-|\gamma(z)|^2)^2  \leq \frac{25}{\pi}(\pi-\Area(F_g))=o_g(1)
$$
so that the conclusions of the proofs of Propositions~\ref{uubf-0} and \ref{ulbfps} become
$$
|\mu_g(z)| \leq \frac{1}{4} + o_g(1) 
$$
and
$$
|\mu_g(0)| \ge \frac{1}{4} - o_g(1) 
$$
respectively.  We hope to take up the thread of the agreement of these asymptotic bounds in a subsequent work.
\end{remark}

\section{Holomorphic lines with uniform negatively pinched Weil-Petersson holomorphic sectional curvatures.} \label{sec:curv of hol lines}

We are now ready to prove Theorem~\ref{ubfhsc} from the introduction.  One slight notational refinement has occurred since that statement: the holomorphic section $\mu_X$ in the statement of Theorem~\ref{ubfhsc} is now more appropriately referred to as $\mu_g \in T_{X_g}\Teich(S_g)$, as defined as in equation (\ref{tps}). We continue to consider $X_g$ to be a sufficiently thick surface as in section~\ref{hmsc}.



\begin{proof}[Proof of Theorem \ref{ubfhsc}]
We wish to prove that the holomorphic sectional curvature $K(\mu_g)$ satisfies
\[ K(\mu_g) \leq  \frac{-81C_0}{6400\cdot \pi^2}<0\]
where $C_0=\frac{2}{3C(1)^2\cdot Vol_{\D}(B(0;1))}.$

Recall $z=x+\textbf{i} y$. Our plan is to bound $||\mu_g||^2_{WP} $ from above, $\sup_{z\in X_g}|\mu_g(z)|$ from below and then apply Proposition~\ref{ratio} to bound the holomorphic sectional curvature.  First,
\begin{align} \label{eqn:Beltrami L2 estimate}
||\mu_g||^2_{WP} &= \int_{F_g}\mu_g\cdot \overline{\mu_g}\rho(z)dxdy \\
&\leq \sup_{z\in F_g}|\mu_g(z)|\cdot \int_{F_g}|\overline{\mu_g}|\rho(z)dxdy  \notag \\ 
&\leq \frac{5}{16} \int_{F_g}| \overline{\mu_g}|\rho(z)dxdy \text{ by Proposition~\ref{uubfps}, since $X_g$ is sufficiently thick } \notag \\
&\leq  \frac{5}{16} \int_{F_g}{\sum_{\gamma \in \Gamma_g} |\gamma'(z)|^2}dxdy \text{ by } \eqref{tps} \notag \\
&\le \frac{5}{16} \int_{\D}{dxdy} \notag \\
&=\frac{5}{16}\pi.\notag
\end{align}

Second, by Proposition \ref{ulbfps}, again using that $X_g$ is sufficiently thick, 
\begin{eqnarray*} \label{eqn:Beltrami sup norm estimate}
\sup_{z\in X_g}|\mu_g(z)|=\sup_{z\in F_g}|\mu_g(z)|\geq |\mu_g(0)|\geq \frac{3}{16}.
\end{eqnarray*}
Finally, we recall from Proposition \ref{ratio} the bound
\begin{equation}\label{eqn:K29}
K(\mu_g)\leq - \frac{C_0\cdot \sup_{z\in X}|\mu_g(z)|^4}{||\mu_g||_{WP}^4}. 
\end{equation}
Thus substituting \eqref{eqn:Beltrami L2 estimate} and \eqref{eqn:Beltrami sup norm estimate} into the right hand side of \eqref{eqn:K29}, we find
$$
K(\mu_g) \leq  \frac{-81C_0}{6400\cdot \pi^2},
$$
as desired.
\end{proof}

Our final task is to prove Corollary~\ref{cor:bdd curv on seq}.

\begin{proof}[Proof of Corollary~\ref{cor:bdd curv on seq}]
We begin with the left hand side. For any $\sigma_g \in T_{X_g}\Teich(S_g)$, Proposition~\ref{lnubw} asserts that the normalized $L^{\infty}$ norm of the Beltrami differential $\sigma_g$ satisfies
\begin{equation*} 
\frac{\sup_{z\in X_g}|\sigma_g(z)|)^2}{\int_{X_g}|\sigma_g(z)|^2 dA(z)} \le C(\inj(X_g))
\end{equation*}
where $C(\inj(X_g))$ is the constant defined in \eqref{defn:C-inj}. We then substitute this estimate into the bound in Proposition~\ref{ratio} to obtain the required lower bound. If, as we did at the end of section~\ref{hmsc}, we considered the setting of a sequence $X_g$ of surfaces whose injectivity radii grow without bound, we could let $\inj(X_g) \to \infty$ so that $C(\inj(X_g)) \to \frac{3}{4\pi}$: we would then obtain the lower bound
$$
-\frac{2}{\pi}\leq K(\sigma_g).
$$
Since $\sigma_g$ is arbitrary in $T_{X_g}\Teich(S_g)$, the conclusion follows.

The upper bound is a direct consequence of the statement of the upper bound in Theorem~\ref{ubfhsc}, as that theorem only required sufficiently large thickness, which is immediately satisfied on our sequence $X_g$, once $g$ is sufficiently large.
\end{proof}

\begin{remark}
Fix two positive numbers $a$ and $b$. 
Let $H(\D)$ denote the holomorphic functions on the disk $\D$, and set

\[H_{a}^{b}(\D)=\{f(z) \in H(\D); \ \sup_{z\in \D}|f(z)|\leq b \ \ \ and \ \ \ |f(0)|\geq a\}.\]

Let $f(z)\in H_{a}^b(\D)$ and consider the image of the Theta-operator $\Theta(f)(z)\in H^0(X_g, K^2)$. Using the same argument in the proof of Theorem~\ref{ubfhsc}, we find

\begin{theorem}\label{soncs}
Let $X_g$ be a sequence of hyperbolic surfaces with $\inj(X_g)\to \infty$ as $g \to \infty$. Then there exists a uniform negative constant $B(a,b)$ such that for $g$ large enough, the Weil-Petersson holomorphic sectional curvature satisfies
\[ \frac{-2}{\pi}\leq K(\mu_g(f)(z)) \leq B(a,b)<0\]
where $\mu_g(f)(z)=\frac{\overline{\Theta(f)(z)}}{\rho(z)}$ and $f$ is arbitrary in $H_{a}^{b}(\D)$.
\end{theorem}
\end{remark}

\begin{remark}
While our focus in this part of the paper was on the Beltrami differential $\mu_g$ defined in terms of the $\Theta$-operator applied to the constant function $f \equiv 1$, we might equally well  have applied the $\Theta$-operator to the functions $f_n(z)=(1-z^n) \ (n \geq 1)$ and obtained similar estimates since $f_n(z)\in H_{1}^{2}(\D)$ for all $n \geq 1$.  It is clear that the convex hull of $\{f_n(z)\}_{n \geq 1}$ is still contained in $H_{1}^{2}(\D)$, \textsl{It is then of interest to know the Weil-Petersson geometry of the image of this convex hull under the $\Theta$-operator.}
\end{remark}

\end{document}